\documentclass[preprint,12pt]{elsarticle}
\usepackage{amsmath}
\usepackage{amsthm}
\usepackage{thmtools}
\usepackage{thm-restate}
\usepackage{amssymb}
\usepackage{multicol}
\usepackage{amsfonts}
\usepackage{mathtools}
\usepackage{bbold}
\usepackage{dsfont}
\usepackage{microtype}
\usepackage{enumitem}
\usepackage{hyperref}
\usepackage{booktabs}
\usepackage[lined,ruled]{algorithm2e}
\usepackage[normalem]{ulem}

\usepackage{tikz}
\usetikzlibrary{fadings}
\usetikzlibrary{shadows.blur}
\usetikzlibrary{arrows.meta}
\usetikzlibrary{shapes}
\usetikzlibrary{calc,angles,quotes}

\theoremstyle{plain}
\newtheorem{theorem}{Theorem}[section]
\newtheorem{proposition}[theorem]{Proposition}

\newtheorem{lemma}[theorem]{Lemma}
\newtheorem{corollary}[theorem]{Corollary}
\theoremstyle{definition}
\newtheorem{definition}[theorem]{Definition}

\theoremstyle{remark}
\newtheorem{remark}[theorem]{Remark}

\newcommand{\R}{\mathbb{R}}   
\newcommand{\B}{\mathbb{B}}   
 
\newcommand{\T}{\mathbb{T}}   
\newcommand{\Sset}{\mathbb{S}}   
\newcommand{\Pset}{\mathds{P}}

\newcommand{\0}{\mathbf{0}} 
\DeclareMathOperator{\activ}{\mathbf{A}}
\DeclareMathOperator{\gram}{\mathbf{G}}

\DeclareMathOperator*{\argmax}{arg\,max}

\DeclareMathOperator{\rem}{rem}
\DeclareMathOperator{\pspan}{pspan}
\DeclareMathOperator{\Diag}{Diag}
\DeclareMathOperator{\diag}{diag}
\DeclareMathOperator{\rank}{rank}
\DeclareMathOperator{\conv}{conv}
\DeclareMathOperator{\cm}{cm}
\DeclareMathOperator{\interior}{int}
\newcommand{\cv}{c\mathbf {V}}

\usepackage{amssymb}
\usepackage{amsmath}

\journal{EURO Journal on Computational Optimization}

\begin{document}

\begin{frontmatter}



\title{On the Computation of Cosine Measures in High Dimensions} 


\author{Warren Hare, Scholar Sun} 

\affiliation{organization={University of British Columbia, Department of Mathematics},
            addressline={3333 University Way}, 
            city={Kelowna},
            postcode={V1V 1V7}, 
            state={British Columbia},
            country={Canada}}

\begin{abstract}
In derivative-free optimization, the cosine measure is a value that often arises in the convergence analysis of direct search methods. 
{Given the increasing interest in high-dimensional derivative-free optimization problems, it is valuable to compute the cosine measure in this setting; however, it has recently been shown to be NP-hard.} We propose a new formulation of the problem and heuristic to tackle this problem in higher dimensions and compare it with existing algorithms in the literature. In addition, new results are presented to facilitate the construction of sets with specific cosine measures, allowing for the creation of a test-set to benchmark the algorithms with.
\end{abstract}


\begin{highlights}
\item Proposed a new reformulation of the cosine measure problem as a norm maximization problem.
\item Based on the reformulation, proposed two new heuristics to solve the cosine measure problem.
\item Derived new theorems to generate sets with nearly arbitrary cosine measure.
\item Compiled a test collection of sets with known cosine measure to benchmark the algorithms.
\item Produced numerical results comparing the algorithms and demonstrated the efficacy of the proposed heuristics. 
\end{highlights}


\end{frontmatter}



\section{Introduction}
Derivative-free optimization (DFO) is the study of optimization algorithms without the use of derivative information \cite{audet2017derivative}. These algorithms are useful when the first-order information is impossible or computationally expensive to obtain. For example, when the objective function is given by the result of a simulation, it is often difficult to recover an analytic form for the gradient. These types of problems are ubiquitous and arise in various fields, including computer science, chemistry, physics, engineering, and geosciences \cite{directsearchSurvey2021}. 

One main sub-class of algorithms used in DFO is direct search methods. At their core, these methods iteratively update an incumbent solution based on the values of a set of trial points, which are computed based on a set of search directions. If none of the trial points improve on the incumbent solution, a step-size parameter is updated, and a new set of trial points is examined. 
{For example, algorithms such as generalized pattern search \cite{torczon_convergence_1997} and mesh adaptive direct search \cite{audet2006mads} fit such a framework. Direct search methods are widely used in practice and we refer the reader to \cite{directsearchSurvey2021} for a survey detailing the numerous applications solved by these methods.}

Typically, the set of search directions are chosen such that they are a \textit{positive spanning set}, which we define in Section~\ref{sec:preliminaries}. This property ensures that the set of search directions can adequately span the space.
The {\em cosine measure} quantifies the density of the search directions and the convergence analysis of direct search methods is often related to this value.
{We also note that the cosine measure problem is closely related to the spherical covering problem in discrete geometry \cite{dodangeh_optimal_2016} and discrepancy theory \cite{sphericalDiscJones2020}.}
To compute the cosine measure requires solving a minimax optimization problem, which we present in Section~\ref{sec:preliminaries}. Recent work in \cite{hare_deterministic_2020} produced an algorithm to compute the cosine measure of positive spanning sets through an exhaustive and deterministic approach. At around the same time, \cite{regis_properties_2021} introduced the theory that characterizes solutions to the cosine measure problem and an outline of an algorithm to solve the cosine measure problem for finite, arbitrary sets. Work done in \cite{sphericalDiscJones2020} proved that the cosine measure problem for positive spanning sets is APX hard, in other words NP-hard to obtain an approximate solution within a constant factor of the true solution. 

Inspired by interest in high-dimensional DFO, we consider the challenge of computing the cosine measure in the high-dimensional setting. To this end, we propose a reformulation of the cosine measure problem as a quadratic program, more specifically as the maximization of the squared 2-norm over a polytope. From this reformulation, we derive an improved exhaustive search method based on enumerating all possible extreme points of the feasible region. In addition, we propose a simple heuristic that allows us to significantly increase the number of solvable dimensions. To test the algorithms' performance, we propose a test set composed of various positive spanning sets with known cosine measures. To ensure a robust test set, two new theorems on the construction of sets with a nearly arbitrary choice of cosine measure are presented.

The rest of the paper is organized as follows. In Section~\ref{sec:preliminaries}, we introduce our notation, definitions, and preliminary theory. In Section~\ref{sec:reformulation}, we present reformulations of the cosine measure problem. In Section~\ref{sec:algos}, we discuss algorithms for solving the cosine measure problem, including a review of existing approaches and the presentation of our new algorithm. In Section~\ref{sec:testset}, we outline how the test set is created and introduce two new ways to construct sets with known cosine measures. In Section~\ref{sec:numerics}, we present our numerical results, comparing the performance between algorithms using our proposed test set. Finally, in Section~\ref{sec:conclusion}, we summarize our results and present possible future directions. 

\section{Preliminaries}\label{sec:preliminaries}

We use $e^{m \times n}$ to denote an $m \times n$ matrix where all entries are 1. We use $e$ to denote a vector where all entries are 1 and $e_i$ to denote an vector where entry $i$ is $1$ and all other entries are $0$. We use $\mathbf 0$ to denote a vector where all entries are 0. Let $A \in \R^{m \times n}$ and $v \in \R^n$. Then $[A]_{i,j}$ denotes the entry in row $i$ and column $j$ of matrix $A$ and $[A]_j$ denotes column $j$ of $A$. We use $[v]_i$ to denote entry $i$ of vector $v$. 
We use $\Diag(v)$ to denote an $n \times n$ matrix with diagonal equal to $v$ and all other entries equal to $0$. If $V_1, \dots, V_k$ are matrices, then $\Diag(V_1, \dots, V_k)$ refers to the block diagonal matrix with blocks formed by $V_i$. Similarly, we use $\diag(V_1)$ to denote the vector where entry $i$ is $[V_1]_{i,i}$. We use $\|v\|$ to denote the usual $2$-norm of $v$. Let $S$ be a set. We use $|S|$ to denote the cardinality of $S$. A polyhedron is the intersection of finitely many half-spaces and a polytope is a bounded polyhedron.

\subsection{Positive Spanning Sets and Cosine Measure} \label{sec:prelim:positive-bases}

We begin by defining the positive span of a set, which is also known as the conical hull of a set. Moreover, we present two characterizations and a basic property of positive spanning sets. 

\begin{definition}
The \emph{positive span} of $\Sset = \{d_1, \dots, d_k\} \subset \mathbb R^n$ is given by
$$\pspan (\Sset) = \{\lambda_1 d_1 + \cdots + \lambda_k d_k  : \lambda_i \geq 0 \text{ for all }i=1, \dots, k\}.$$
We say a set $\Sset$ is \emph{positive spanning} if $\pspan(\Sset) = \R^n$.
\end{definition}

\begin{proposition} \cite[Proposition 2.1 iv]{regis_properties_2016} \label{thm:pos-scale-spanning-set}
    If $\Sset = \{d_1,\dots, d_k\} \subset \R^n$ and $\lambda_1, \dots, \lambda_k > 0$, then $\pspan(\Sset) = \pspan(\{\lambda_1d_1,\dots, \lambda_k d_k\})$.
\end{proposition}

\begin{theorem}\label{thm:pspan-char}
    Suppose  $\Sset = \{d_1, \dots, d_k\} \subset \R^n$. Then $\Sset$ positively spans $\R^n$ if and only if 
    \begin{enumerate}[label={(\roman*)}]
        \item \cite[Theorem 2.6]{regis_properties_2016} for every $v \in \R^n \setminus \{\0\}$, there exists an index $i \in \{1, \dots, k\}$ such that $v^\top d_i > 0$. \label{thm:pspan-char:spanning-direction}
        \item \cite[Theorem 2.3]{DFO-textbook-2009}\label{thm:pos-spanning-sum-equals-0} $\Sset$ linearly spans $\R^n$ and there exist real scalars $\lambda_1,...,\lambda_k$ with $\lambda_i>0$, $i\in \{1,\dots, k\}$, such that $\sum_{i=1}^k \lambda_id_i =0$. \label{thm:pspan-char:pos-spanning-sum-equals-0}

    \end{enumerate}
\end{theorem}

Now we introduce the cosine measure. Intuitively, it quantifies the largest gap in a set of vectors: a higher value indicates a smaller gap, while a lower value corresponds to a larger gap.

\begin{definition}
    Let $\Sset = \{d_1, \dots, d_k\} \subset \R^n \setminus \{\0\}$. The \textit{cosine measure of} $\Sset$ is 
    $$\cm(\Sset) =  \min_{v \in \R^n, \|v\|=1} \cm (\Sset, v)= \min_{v\in\R^n, \|v\|=1}  \max_{d \in \Sset} \frac{d^\top v}{\|d\|}.$$
    The solution(s) $v \in \R^n$ to this minimax problem are called the \textit{cosine vectors}. The set of all such solutions is the \textit{cosine vector set} $\cv(\Sset)$. If $S$ is a matrix, then $\cm(S)$ refers to $\cm(\Sset)$ where $\Sset$ is the set of all the columns of $S$.  
\end{definition}

\begin{lemma}\cite[Theorem 4.1]{regis_properties_2021}\label{thm:existence-cmp}
    Let $\Sset \not = \varnothing$. Then there exists a minimizer of the cosine measure problem. 
\end{lemma}

Related to a cosine vector is the active set, that is the subset of elements which make the optimal cosine with the cosine vector.

\begin{definition} Let $\Sset$ be a positive spanning set of $\R^n$ and let $u^* \in \cv (\Sset)$. The \textit{active set} of $\Sset$ with respect to $u^*$, denoted $\activ (\Sset,u^*)$ is defined as 
$$\activ (\Sset,u^*) = \left\{ \frac{d}{\|d\|} \in \R^n : d \in \Sset, \frac{d^\top u^*}{\|d\|} = \cm (\Sset) \right\}.$$
\end{definition}

Motivated by the fact that the cosine vector must make equal an angle with all vectors in the active set, we define the uniform angle subspace to be the set of all vectors that make an equal angle with a given set.
\begin{definition}
\cite[Theorem 5.1]{regis_properties_2021} Let $\Sset=\{d_1, \dots, d_k\} \subset \R^n \setminus \{\mathbf 0\}$ with $k \geq 2$ and let $\theta(v_1,v_2)$ denote the angle between nonzero vectors $v_1$ and $v_2$ in $\R^n$. Define 
$$\Theta(\Sset) = \{0\} \cup \{v \in \R^n | \theta(v, d_1) = \theta(v,d_2) = \dots = \theta(v,d_k)\}.$$
Then, given a set $\Sset$ with at least two nonzero vectors in $\R^n$ we refer to $\Theta(\Sset)$ as the {\em uniform angle subspace.}
\end{definition}

Next we define a minimum cosine cone and specialize it to work with the uniform angle subspace of a set $\Sset$. 
\begin{definition}\label{defn:min-cosine-cone}
    Let $y \in \mathbb R^n$, $y \not = \mathbf 0$, and $-1 \leq \alpha \leq 1$. Define
    $$K^\circ_\alpha(y) = \left\{v \in \mathbb R^n : v \not = \mathbf 0 \text{ and } \frac{v^\top y}{\|v\|\|y\|} > \alpha \right\} \cup \{\mathbf 0\}.$$
\end{definition}
\begin{definition}\label{defn:cone}
    Let $\Sset=\{d_1, \dots, d_k\}\subset \R^n$ be a set of unit vectors such that $\Theta(\Sset) \not = \{\mathbf 0\}$. Let $y \in \Theta(\Sset)$ and $\|y\|=1$. Define $$K^\circ(\Sset,y) = K^\circ_{\alpha_\Sset}(y)$$ where $\alpha_\Sset = d^\top_i y$ for all $i=1, \dots, k$.
\end{definition}

We end the section with a useful theorem that relates the positive spanning property with the cosine measure of a set. 

\begin{theorem}\cite[Theorem 4.2 (iv)]{regis_properties_2021}\label{thm:pos-span-cm>0}
    Let $\Sset = \{d_1, \dots, d_k\} \subset \R^n \setminus \{\mathbf 0\}$. Then $\Sset$ positively spans $\R^n$ if and only if $\cm(\Sset) > 0$.
\end{theorem}
\subsection{Gram Matrices} \label{sec:prelim:gram-matrices}

This section examines results related to the Gram matrix, a useful tool for discussing the relationship between sets and cosine measures. For example, it was used in \cite{naevdal_positive_2019} to prove certain sets had optimal cosine measure, in \cite{hare_deterministic_2020} to characterize solutions to the cosine measure problem, and in \cite{hare_nicely_2023} to compute the cosine measure of a specific set.

\begin{definition}
    Let $S = [d_1 \ d_2 \dots d_k] \in \R^{n \times k}$. The \textit{Gram matrix} of $S$, denoted by $\gram (S) \in \R^{k \times k}$, is defined to be $\gram(S) = S^\top S$.
\end{definition}

A special case arises when $S \in \R^{n \times n}$ and is invertible. We define a {\em Gram vector} and {\em Gram value}, which can be derived from the Gram matrix.

\begin{definition}
    Let $\B = \{d_1, \dots, d_n\}\subset \R^n$ be a basis. We define the \textit{Gram vector} $u\in \R^n$ of $\B$ with \textit{Gram value} $\gamma_\B$ to be a unit vector such that $u^\top  d_i = \gamma_\B >0$ for all $i \in \{1, \dots, n\}$.
\end{definition}

The following results prove the existence and uniqueness of the \textit{Gram vector} and \textit{Gram value}, which are derived from the Gram matrix. 
\begin{lemma}\label{thm:reordering}
    Let the columns of $B \in \R^{n\times n}$ form a basis. Then for any orthogonal matrix $P$ we have $$\frac{1}{\sqrt{e^\top \gram(BP)^{-1} e}} = \frac{1}{\sqrt{e^\top P^\top\gram(B)^{-1}P e}}.$$ 
\end{lemma}
\begin{proof} Since $P$ is orthogonal, 
    $$e^\top \gram (BP)^{-1} e = e^\top ((BP)^\top BP)^{-1} e\
        = e^\top P^\top (B^\top B)^{-1}P e = e^\top P^\top \gram (B)^{-1} P e.$$
\end{proof}
\begin{theorem}\label{thm:gram-vector}
    Let $B = [d_1 \ d_2 \dots d_n] \in \R^ {n \times n}$ where the columns are linearly independent unit vectors. Let $\B = \{d_1, \dots, d_n\}$ denote the set of columns. Then there exists a unique Gram vector of $\B$ with Gram value
    \begin{align*}
        \gamma_\B = \frac{1}{\sqrt{e^\top \gram (B)^{-1}e}}.
    \end{align*}
\end{theorem}
\begin{proof}
    The existence and uniqueness of the Gram vector is proven in \cite[Lemma 1]{naevdal_positive_2019}. To see that $\gamma_\B$ is well-defined, it suffices to note that, by Lemma~\ref{thm:reordering}, $u$ and $\gamma_\B$ are invariant under the reordering of columns in the matrix $B$.
\end{proof}

The following results build up to Corollary~\ref{cor:necessary-cond-gram-vector}, which provide necessary conditions for a vector to belong to the cosine vector set. 

\begin{lemma}\cite[Lemma 13]{hare_deterministic_2020}\label{theorem:all-inner-prod-equal-implies-gram}
     Let $\B = \{d_1, \dots, d_n\} \subset \R^n$ be a basis of unit vectors. Suppose $u$ is a unit vector such that $u^\top d_1 = \cdots = u^\top d_n = \alpha >0$. Then $\alpha = \gamma_\B$, where $\gamma_\B$ is defined in Lemma~\ref{thm:gram-vector}.
\end{lemma}

\begin{lemma} \cite[Corollary 18]{hare_deterministic_2020}\label{thm:active-set-has-basis}
    Let $\Sset = \{d_1, \dots, d_k\} \subset \R^n$ be a positive spanning set of unit vectors and $u^* \in c \mathbf V (\Sset)$. Then $\mathbf A(\Sset, u^*)$ contains a basis of $\R^n$.
\end{lemma}

\begin{corollary}\label{cor:necessary-cond-gram-vector}
     Let $\Sset=\{d_1, \dots, d_k\} \subset \R^n$ be a positive spanning set of unit vectors and $u^* \in \cv(\Sset)$. Then, there exists a basis $\B \subset \Sset$ such that $u^*$ is the Gram vector of $\B$ with a Gram value equal to $\cm(\Sset)$. 
\end{corollary}
\begin{proof}
   {Combine Lemma~\ref{theorem:all-inner-prod-equal-implies-gram} and Lemma~\ref{thm:active-set-has-basis}}. 
\end{proof}

Lastly, to help us compute the inverse of Gram matrices, we use a rank-1 update formula.
\begin{theorem} \cite[Equation (2)]{rank1update}
    Let $A \in \R^{n \times n}$ be an invertible matrix and $u,v \in \R^n$. Then $A+uv^\top$ is invertible if and only if $1 + v^\top A^{-1}u\not = 0$. In particular, 
    \begin{align*}
        (A + uv^\top)^{-1} = A^{-1} - \frac{A^{-1}uv^\top A^{-1}}{1+v^\top A^{-1}u}.
    \end{align*}
\end{theorem}

\section{Reformulation}\label{sec:reformulation}

This section introduces two reformulations of the cosine measure problem. The first reformulation was introduced in \cite{regis_properties_2021} and shows the cosine measure problem can be equivalently formulated as a quadratically constrained linear program. The second reformulation requires that $\Sset$ is a positive spanning set, and shows the cosine measure problem can be written as a convex maximization problem over a polytope. 
\begin{theorem}\cite[Section 8.1]{regis_properties_2021}
Let $\Sset$ be a finite set of unit vectors. The optimization problem
\begin{align}
\begin{split}
    \min_{v\in\R^n, \|v\| = 1} \max_{d \in \Sset}\ d^\top v
\end{split} \label{eqn:CMP}
\end{align}
is equivalent to
\begin{align}
\begin{split}
    \underset{y\in \R^n, z \in \R}{\min} \quad & z  \\
    \text{subject to} \quad &  y^\top  d \leq z , \quad \forall d \in \Sset\\
    & \|y\|^2 = 1
\end{split} \label{eqn:QCLP}
\end{align}
are equivalent.
\end{theorem}
We call problem \eqref{eqn:QCLP} the quadratically constrained linear program (QCLP) formulation.
\begin{theorem}\label{thm:reformulation}
Let $\Sset$ be a finite, positive spanning set of unit vectors. Then, optimization problem \eqref{eqn:CMP} is equivalent to
\begin{align}
\begin{split}
    \underset{x\in \R^n}{\max} \quad & \|x\|^2  \\
    \text{subject to} \quad & x \in P\\ 
     & P = \{x \in \R^n: x^\top d \leq 1, d\in \Sset\}.
\end{split} \label{eqn:QP-polytope}
\end{align}
\end{theorem}
We call problem \eqref{eqn:QP-polytope} the quadratic program (QP) formulation. Before proving Theorem~\ref{thm:reformulation}, we first prove the existence of a maximizer for problem \eqref{eqn:QP-polytope} and that having $\Sset$ be positive spanning is a necessary condition for the existence of a solution. 
\begin{lemma} \label{thm:existence-polytope}
    Let $\Sset$ be a finite set of unit vectors. If $\Sset$ is a positive spanning set, then $P = \{x \in \R^n: x^\top d \leq 1, d\in \Sset\}$ is a polytope. Moreover, there exists a maximizer for problem \eqref{eqn:QP-polytope}.
\end{lemma} 
\begin{proof}
    It is immediately seen that $P$ is a polyhedron as it is the intersection of finitely many closed half-spaces. Evidently, $\mathbf 0 \in P$ since $d^\top \mathbf 0 = 0 \leq 1$ for all $ d \in \Sset$. If $P$ is unbounded, then there must exist a nonzero $x \in \R^n$ such that $\lambda x \in P$ for all $\lambda \geq 0$.  By Theorem~\ref{thm:pspan-char}~\ref{thm:pspan-char:spanning-direction} there exists some $d \in \Sset$ such that $ x^\top d > 0$. Take $\lambda$ to be sufficiently large such that $\lambda x^\top d > 1$, which contradicts $\lambda x ^\top d \leq 1$. Thus $P$ is compact and by the extreme value theorem the maximum is attained. 
\end{proof}
Recall that, if $f$ is strictly convex, then $x^* \in \argmax_{x \in P}{f}$ implies that $x^*$ is necessarily a vertex of $P$. 
Now we introduce a theorem that relates the feasible solutions of problem \eqref{eqn:QCLP} and problem \eqref{eqn:QP-polytope}.
\begin{theorem}\label{thm:both-feasible} The following hold:
    \begin{enumerate}[label={(\roman*)}]
        \item If $(y,z)$ is feasible for problem $\eqref{eqn:QCLP}$, then $\frac{y}{z}$ is feasible for problem \eqref{eqn:QP-polytope},\label{thm:both-feasible:forward}
        \item If $x \not = \mathbf {0}$ is feasible for problem \eqref{eqn:QP-polytope}, then $\left(\frac{x}{\|x\|},\frac{1}{\|x\|}\right)$ is feasible for problem \eqref{eqn:QCLP}. \label{thm:both-feasible:backward}
    \end{enumerate}
\end{theorem}
\begin{proof}
    \ref{thm:both-feasible:forward} Suppose $(y, z)$ is feasible for problem \eqref{eqn:QCLP}. Then for all $d \in \Sset$, 
    \begin{align*}
        \left(\frac{y}{z}\right)^\top d = \frac{1}{z} (y^{\top} d) \leq \left(\frac{1}{z}\right)z = 1
    \end{align*}
    so $\frac{y}{z}$ is feasible for problem $\eqref{eqn:QP-polytope}$. \\ \ref{thm:both-feasible:backward}  Suppose $x \not= \mathbf 0$ is feasible for problem \eqref{eqn:QP-polytope}. Then $\left \|\frac{x}{\|x\|}\right\|^2 = 1 $ and for all $d \in \Sset$
        \begin{align*}
            x^{\top} d \leq 1 \implies \left( \frac{x}{\|x\|}\right)^\top d \leq \frac{1}{\|x\|} 
        \end{align*}
    so $\left(\frac{x}{\|x\|},\frac{1}{\|x\|}\right)$ is feasible for problem \eqref{eqn:QCLP}.
\end{proof}
Finally, we move onto the proof of Theorem~\ref{thm:reformulation}. 
\begin{proof}
$[\eqref{eqn:QCLP} \Rightarrow \eqref{eqn:QP-polytope}]$ Let $(y^*, z^*)$ be an optimal solution of problem \eqref{eqn:QCLP}. Since $\Sset$ is positive spanning, we have $z^* > 0$ by Theorem~\ref{thm:pos-span-cm>0}. Define $x^* = \frac{y^*}{z^*}$. By Theorem~\ref{thm:both-feasible}, $x^*$ is feasible for problem \eqref{eqn:QP-polytope}. To show optimality, let $x \in P$, $x \not = \mathbf 0$ be arbitrarily chosen. Note that by Theorem~\ref{thm:both-feasible}, $\left(\frac{x}{\|x\|},\frac{1}{\|x\|}\right)$ is feasible for problem \eqref{eqn:QCLP}. By the optimality of $z^*$ we have
    \begin{align*}
    z^* \leq \frac{1}{\|x\|} &\implies
    \frac{\|y^*\|}{z^*} \geq \|x\| \implies \left\|\frac{y^*}{z^*}\right\| \geq \|x\|  
    \implies \|x^*\|^2 \geq \|x\|^2 .
    \end{align*}
    If $x= \mathbf 0$, the result is immediate since $\|x^*\|^2 \geq \|x\|^2 = 0$. Thus $x^*$ is an optimal solution for problem \eqref{eqn:QP-polytope}. \\
    $[\eqref{eqn:QP-polytope}\Rightarrow \eqref{eqn:QCLP}]$ Let $x^*$ be an optimal solution of problem \eqref{eqn:QP-polytope}, which exists by Lemma~\ref{thm:existence-polytope}. Note $x^*$ is necessarily nonzero. 
    Define $(y^*,z^*) = \left(\frac{x^*}{\|x^*\|}, \frac{1}{\|x^*\|}\right)$.
    By Theorem~\ref{thm:both-feasible},
    $(y^*,z^*)$ is feasible for problem \eqref{eqn:QCLP}. To show optimality, let $(y,z)$ be an arbitrary feasible solution to problem \eqref{eqn:QCLP}, where we note $z>0$ by Theorem~\ref{thm:pos-span-cm>0}. By Theorem~\ref{thm:both-feasible}, $\frac{y}{z}$ is feasible for problem \eqref{eqn:QP-polytope}. By the optimality of $x^*$,
    \begin{align*}
       \|x^*\|^2 \geq \left\|\frac{y}{z}\right\|^2 \implies \left\|\frac{y^*}{z^*}\right\| \geq \left\|\frac{y}{z}\right\| \implies \frac{1}{z^*} \geq \frac{1}{z} \implies z^* \leq z
    \end{align*}
    so $(y^*,z^*)$ is an optimal solution for problem \eqref{eqn:QCLP}. 
\end{proof}

Observe that the vertices of $P$ are candidate solutions of problem \eqref{eqn:QP-polytope} and that Gram vectors are candidate solutions of the cosine measure problem by Corollary~\ref{cor:necessary-cond-gram-vector}. The following theorem provides a relationship between the two objects.

\begin{theorem} \label{thm:vertex-to-gram}
     Let $\Sset$ be a finite positive spanning set and $$P = \{x \in \R^n : x^\top d \leq 1, d \in \Sset \}.$$ 
    \begin{enumerate}[label={(\roman*)}]
        \item Let $x$ be a vertex of $P$, $\T = \{d \in \Sset: x^\top d = 1\}$, and $\B \subseteq \T$ be a basis of $\R^n$. Then $\frac{x}{\|x\|}$ is a Gram vector of $\B$ with Gram value $\frac{1}{\|x\|}$ and $K^\circ \left(\B,\frac{x}{\|x\|}\right) \cap \Sset = \varnothing$. \label{thm:vertex-to-gram-forward}
        \item Let $\B \subseteq \Sset$ be a basis and $u$ be the Gram vector of $\B$ with Gram value $\alpha$. If $K^\circ \left(\B, u \right) \cap \Sset = \varnothing$, then $\frac{u}{\alpha}$ is the unique vertex where $\left(\frac{u}{\alpha}\right)^\top d = 1$ for all $d \in \B$. \label{thm:vertex-to-gram-backward}
    \end{enumerate}
\end{theorem}

\begin{proof}
\ref{thm:vertex-to-gram-forward} Let $x$ be a vertex of $P$ and $\B \subseteq \T$ be a basis of $\R^n$. Observe that $d^\top \left(\frac{x}{\|x\|}\right) = \frac{1}{\|x\|}$ for all $d \in \B$, so by Lemma~\ref{theorem:all-inner-prod-equal-implies-gram}, $\frac{x}{\|x\|}$ is the Gram vector of $\B$ with Gram value $\frac{1}{\|x\|}$. Moreover, since $d^\top \left(\frac{x}{\|x\|}\right) \leq \frac{1}{\|x\|}$ for all $d \in \Sset$, by definition we have $K^\circ \left(\B, \frac{x}{\|x\|}\right) \cap \Sset = \varnothing$.\\
\ref{thm:vertex-to-gram-backward} Let $\B \subseteq \Sset$ be a basis and $u$ be the Gram vector of $\B$ with Gram value $\alpha$ and $K^\circ \left(\B, u \right) \cap \Sset = \varnothing$. Suppose $\frac{u}{\alpha} \not \in P$. Then there exist $d \in \Sset \setminus \B$ such that 
\begin{align*}
    d^\top \left(\frac{u}{\alpha}\right) > 1 \implies d^\top u > \alpha \implies d \in K^\circ (\B, u) \cap S \not = \varnothing,
\end{align*}
so we can conclude $\frac{u}{\alpha} \in P$. Moreover, $\left(\frac{u}{\alpha}\right) ^\top d = \frac{\alpha}{\alpha}= 1$ for all $d \in \B$, so we have $n$ tight linearly independent constraints at $\frac{u}{\alpha}$. Thus we can conclude $\frac{u}{\alpha}$ is a vertex of $P$. The uniqueness of the vertex $\frac{u}{\alpha}$ follows directly from the uniqueness of the Gram vector $u$ with respect to the basis $\B$. 
\end{proof}

\begin{remark}
    We note that in \cite{mangasarian_norm-complexity_1986} it has been shown that the norm-maximization problem over an arbitrary polyhedron is NP-hard. However, the complexity of norm-maximization over the specific polytope defined by $P = \{x \in \R^n: x^\top d \leq 1, d\in \Sset\}$ where $\Sset$ is a positive spanning set is unknown. Recently in \cite{sphericalDiscJones2020}, the cosine measure problem was shown to be NP-hard. Since we have shown equivalence between the two problems, we can conclude that the complexity of norm-maximization over $P$ is also NP-hard. 
\end{remark} 
\section{Algorithms}\label{sec:algos}

In this section, we present the algorithms used to solve the cosine measure problem. We restrict ourselves to positive spanning sets, as it was shown in \cite[Lemma 33] {audet2024cosinemeasurerelativesubspace} that if the set is not positive spanning, the cosine measure problem can be formulated as a second-order cone program, for which there are efficient solvers.

Section $\ref{subsec:algos-basis-enumeration}$ describes an enumerative approach, from \cite{hare_deterministic_2020}, which searches all bases and constructs the corresponding Gram vectors. Section $\ref{subsec:kkt-enumeration}$ describes an enumerative approach, from \cite{regis_properties_2021}, which exhaustively searches all Karush-Kuhn-Tucker (KKT) points. The next three sections present new methods based on the reformulation from Section~\ref{sec:reformulation}. Section~\ref{sec:algo-branch-and-bound} describes how branch-and-bound can be used, Section~\ref{sec:vertex-enumeration} describes an exhaustive approach which enumerates the vertices of a polytope, and Section~\ref{sec:randomized-linear-programs} describes an approach which solves a number of randomized linear programs.

\subsection{Basis Enumeration}\label{subsec:algos-basis-enumeration}
Work done in \cite{hare_deterministic_2020} presented a deterministic method to compute a global solution of the cosine measure problem for finite positive spanning sets. This is due to the insight from Corollary~\ref{cor:necessary-cond-gram-vector}, which says the cosine vector is necessarily a Gram vector of some basis contained in the positive spanning set.
{A natural solution is to enumerate all bases, compute the associated Gram vectors, and return the one yielding the smallest objective value.}

        



        
        


Details for this method can be found in \cite[Algorithm 1]{hare_deterministic_2020}, while \cite[Theorem 19]{hare_deterministic_2020} proves that the method returns the exact value of the cosine measure. 

\subsection{KKT Point Enumeration}\label{subsec:kkt-enumeration}

In a similar fashion, work done in \cite{regis_properties_2021} presents a framework to solve the cosine measure problem for arbitrary finite sets. This is done by enumerating all subsets of size less than or equal to $n$ and checking if the unit vector belonging to the uniform angle subspace generated by the subset is a Karush-Kuhn-Tucker (KKT) point. 

{The proposed outline in \cite{regis_properties_2021} for calculating the cosine measure were based on the derived necessary and sufficient conditions of a KKT point. We refer the reader to the paper for details.
It should be noted that it was not proven that a global minimizer necessarily satisfies the KKT conditions, however, in practice, we have found the algorithm to return the minimizer in all cases. As \cite{regis_properties_2021} does not contain detailed pseudo-code, we provide Algorithm~\ref{alg:kkt-enumeration-regis}.}

\begin{algorithm}[H]\caption{KKT points enumeration algorithm}\label{alg:kkt-enumeration-regis}
    \DontPrintSemicolon
        \KwIn{ Finite set $\Sset$ of $\R^n$ containing $k$ vectors}
        
\For{$m = 2$ \textbf{to} $n+1$}{
    \ForAll{\text{\upshape subsets} $\T = \{d_1, \dots, d_m\} \subseteq \Sset$}{
        
        
        \uIf{$\rank(T)=n$}{ 
        
            $\beta \in \pspan(\T) \cap \Theta(\T)$ 
        
            $x \leftarrow \frac{T \beta }{\| T \beta \|}$
        
            \lIf{$K^\circ(\T,x) \cap \Sset = \varnothing$}{
                $\gamma^1_\T = x^\top d_1$, $u^1_{\T} = x$}
            \lIf{$K^\circ(\T,-x) \cap \Sset = \varnothing$}{
$\gamma^2_\T = -x^\top d_1$, $u^2_{\T} = -x$}
        }\ElseIf{$\rank\left(\begin{bmatrix}T^\top & -e\end{bmatrix}^\top\right) = n$ \textbf{\upshape and} $0 \in \conv(\T)$}{
        
            $x \in \T^\perp$ by SVD

            \lIf{$K^\circ(\T,x) \cap \Sset = \varnothing$}{$\gamma^3_\T = 0$, $u^3_{\T} = x$}
        
            \lIf{$K^\circ(\T,-x) \cap \Sset = \varnothing$}{$\gamma^4_\T = 0$, $u^4_{\T} = -x$}
        }
    }
}

\textbf{return}

\quad \ $\min_{\T \subseteq \Sset, i \in \{1,2,3,4\}} \gamma^i_\T$ : the exact cosine measure\\
\quad \ $\{u^i_\T: \gamma^i_\T = \cm(\Sset)\}$ : the cosine vector set
\end{algorithm}

\subsection{Furthest Vertex via Branch-and-Bound}\label{sec:algo-branch-and-bound}

The first approach to using reformulation \eqref{eqn:QP-polytope} is to apply an off-the-shelf branch-and-bound method, which provides a global solution to non-convex nonlinear problems. These methods are widely available and offered by many commercial solvers, such as CPLEX \cite{CPLEX} and Gurobi \cite{gurobi}. In our case, we employ Gurobi V12.0.1 to perform our benchmarking. 

\subsection{Vertex Enumeration}\label{sec:vertex-enumeration}
The next method to exploit reformulation \eqref{eqn:QP-polytope} is to leverage the fact that the solution must lie at a vertex of $P$. This gives rise to Algorithm~\ref{alg:vertex-enumeration}.

\begin{algorithm}[H]\caption{Furthest Vertex via Vertex Enumeration Algorithm}\label{alg:vertex-enumeration}
    \DontPrintSemicolon
        \KwIn{Finite positive spanning set $\Sset$ of $\R^n$ containing $k$ vectors}

        Enumerate the vertices of $P = \{x:x^\top d \leq 1, \forall d \in \Sset\}$

        $v^* \in \argmax\{\|v\|: v \text{ is a vertex of } P\}$ 
        
        \textbf{return}

        \quad \ $\frac{1}{\|v^*\|}$ : the exact cosine measure\\
        \quad \  $\left\{\frac{v}{\|v\|}: v \text{ is a vertex of } P, \|v\| = \|v^*\|\right\}$ : the cosine vector set
\end{algorithm}

Comparing Algorithm~\ref{alg:vertex-enumeration} to the basis enumeration of Section \ref{subsec:algos-basis-enumeration}
, notice that basis enumeration considers all $\B \subset\Sset$ where $\B$ is a basis. Under our reformulation, this can be interpreted as exhaustively enumerating the extreme points generated by all possible bases $\B$. However, not all extreme points generated by $\B$ are belong to the set $P$, so computations are wasted on points that can never be the solution. 

With this observation in mind, we employ algorithms specialized in enumerating the vertices of a polytope most efficiently. In particular, we consider the lexicographical reverse search method presented in \cite{Avis2000-LRS-lexicographical-reverse-search}.  This approach enumerates only the bases that correspond to the vertices of $P$, ensuring that all considered points are feasible. 

We remark that to solve the cosine measure problem, we are agnostic to the algorithm we choose to enumerate the vertices. However, the number of vertices is exponential with respect to the dimension of the polytope, so this approach is still exhaustive and has exponential complexity. 

\subsection{Furthest Vertex via Randomized Linear Programs}\label{sec:randomized-linear-programs}
Using the knowledge that the solution to reformulation \eqref{eqn:QP-polytope} necessarily lies at a vertex and that solving linear programs (LP) is computationally cheap, we propose the following heuristic. 

We sample a random direction $c$ uniformly on the unit sphere and then solve the following LP,
\begin{align} \label{LP:random-LP}
    \begin{split}
    \max c^\top x \ \ \text{ subject to }\ \ x \in P,
    \end{split}
\end{align}
the solution of which is necessarily some vertex of $P$. We repeat this process many times and return the vertex with the maximum norm. This gives rise to the following algorithm. 

\begin{algorithm}[H]\caption{Random LPs}\label{alg:random-lp}
    \DontPrintSemicolon
        \KwIn{Finite positive spanning set $\Sset$, number of iterations $k$ }
        $v=\0$
        
        \For{$i=1, \dots, k$}{
            Uniformly sample $c \in \{x:\|x\|=1\}$

            Let $\bar v$ be the solution to LP \eqref{LP:random-LP}

            \lIf{$\|\bar v\| > \| v\|$}{
                $v=\bar v$
            }
        }
        \textbf{return}
        
        \quad \ $\frac{1}{\|v\|}$ : an approximation of the cosine measure \\
        \quad \ $\frac{v}{\|v\|}$ : a vector that provides the approximate cosine measure
\end{algorithm}
We now prove that with probability 1, if run {\em ad infinitum}, then Algorithm~\ref{alg:random-lp} will return the exact optimal solution. We first require two lemmas.

\begin{lemma} \label{lemma:basis-existence-to-vertex}
    Let $\Sset$ be a positive spanning set and $$P = \{x\in \R^n: x^\top d \leq 1, d \in \Sset\}.$$ Let $v$ be a vertex of P. Then there exists a basis $\B \subset \Sset$ of $\R^n$ such that for all $c \in \pspan(\B)$ where $\|c\|=1$, the solution to the linear program 
    \begin{align} \label{eq:random-lp-lemma}
        \max c^\top x \text{ subject to } x\in P
    \end{align}
    is the vertex $v$.
\end{lemma}

\begin{proof}
By the definition of a vertex, there exists a basis $\B$ such that $$\B = \{d_1,\dots, d_n\} \subseteq\{d \in \Sset: d^\top v = 1\} .$$ Let $c \in \R^n$ with $c \in \pspan(\B)$ and $\|c\|=1$. Then we may write
\begin{align*}
    c = \lambda_1 d_1 + \dots + \lambda_n d_n    
\end{align*}
with $\lambda_i \geq 0$. Observe that 
\begin{align*}
    c^\top v 
    = \lambda_1 (d_1^\top v) + \dots + \lambda_n (d_n^\top v) 
    = \lambda_1 + \dots + \lambda_n.
\end{align*}
Now let $x \in P \setminus \{v\}$ be arbitrarily chosen. Since $\B$ is a basis, $v$ is the unique vector such that $v^\top d = 1$ for all $d \in \B$. Thus, there exists some $i \in \{1, \dots, n\}$ such that $d_i^\top x < 1$. Hence,
\begin{align*}
    c^\top x = \lambda_1(d^\top_1 x) + \dots + \lambda_n(d^\top_n x) 
    < \lambda_1 + \dots + \lambda_n 
\end{align*}
Thus we conclude that $v$ is a maximizer for \eqref{eq:random-lp-lemma}.
\end{proof}

\begin{lemma}\label{lemma:pos-span-non-empty-interior}
    Let $\B = \{d_1, \dots, d_n\}$ be a basis of $\R^n$ and  
    $T = \pspan(\B) \cap B_1(\0)$.
    Then $\interior(T) \not = \varnothing$.
\end{lemma}
\begin{proof}
    As $\B$ is a basis, $\interior(\pspan(\B)) \neq \varnothing$, hence  $\interior(T)\neq \varnothing$. 
\end{proof}
We now present the proof of convergence of Algorithm~\ref{alg:random-lp}.
\begin{theorem}\label{thm:random-lp-convergence}
    With probability $1$, when run ad infinitum, Algorithm~\ref{alg:random-lp} finds a minimizer. 
\end{theorem}
\begin{proof}
    Let $v^*$ be a maximizer for problem \eqref{eqn:QP-polytope}, which exists by Lemma~\ref{thm:existence-polytope}. By Lemma~\ref{lemma:basis-existence-to-vertex}, there exists a basis $\B \subset \Sset$ such that for all $c \in \pspan(\B)$ the solution to LP $\eqref{LP:random-LP}$
    Since $\interior(\pspan(\B)\cap B_1) \neq \varnothing$,
    when $c$ is uniformly sampled on $S^{n-1}$, there is a constant nonzero probability of $c \in \pspan(\B)$. Consequently, on each iteration there is a constant nonzero probability that the solution to \eqref{LP:random-LP} is $v^*$, and the result follows.
\end{proof}

\section{Test Set}\label{sec:testset}
This section presents a collection of sets that can be used to benchmark the performance of algorithms. {In particular, we focus on the construction of sets that are positive bases
and proving the value of the cosine measure of each set. We say $\Pset$ is a positive basis if it is positive spanning and for all $d \in \Pset$, we have $d \not \in \pspan(\Pset \setminus \{d\})$. A positive basis is minimal if $|\Pset|=n+1$ and maximal if $|\Pset| = 2n$.}

Section~\ref{sec:transformations} introduces two theorems used to transform sets while preserving the cosine measure. {Sections~\ref{sec:minimal-delta-angle-pb} and~\ref{sec:maximal-delta-angle-pb} cover the construction of the minimal and maximal positive bases respectively. Lastly, Section~\ref{sec:intermediate-pbasis} and Section~\ref{sec:random-pss} overview results relating to optimal orthogonal positive bases and random positive spanning sets.}

\subsection{Transformations} \label{sec:transformations}

This section presents transformations that can be applied to a set without changing the cosine measure. These transformations are applied to ensure robustness and to increase the complexity of our test set. In this section, when we consider a set $\Sset = \{d_1, \dots, d_k\}$, we use the fact that it can be represented interchangeably as a matrix $S = [d_1 \dots d_k]$, where each column corresponds to a vector in the set.

\begin{theorem} (Cosine Measure is Rotation and Permutation Invariant) \cite[Theorem 5.1]{regis_properties_2021} \label{thm:rotation-invariant}
    Let $\Sset = \{d_1, \dots, d_k\} \subset \R^n\setminus\{\mathbf 0\}$ and $S$ to be the matrix constructed using the vectors in $\Sset$ as columns. Let $P = \alpha Q$ where $Q \in \R^{n \times n}$ is an orthogonal matrix and $\alpha \not = 0$ is a scalar. Then $\cm(P S) = \cm (S)$. 
\end{theorem} 

Additionally, Theorem~\ref{thm:augment-set-vector} shows how we may augment the set by adding additional vectors to increase the complexity. This increases the number of bases needed to be checked when running exhaustive methods and acts to break the symmetry of the sets.

\begin{theorem}\label{thm:augment-set-vector}
    Let $\Sset$ be a set of unit vectors in $\R^n$ and $u^* \in c\mathbf V(\Sset)$. Let $v \in \R^n$ such that $u^{*\top}v \leq \cm(\Sset)$. Then $\cm(\Sset) = \cm(\Sset \cup \{v\})$.
\end{theorem}
\begin{proof}
First notice that $$\cm(\Sset)=\max_{d\in \Sset} u^{*\top} d = \max_{d\in \Sset\cup \{v\}} u^{*\top} d,$$ where the last equality holds by the assumption $u^{*\top}v \leq \cm(\Sset)$. Next, let $q\in \R^n$ with $\|q\|=1$ and notice that 
$$\cm(\Sset) = \min_{z\in\R^n, \|z\|=1} \max_{d\in \Sset} z^\top d \leq \max_{d \in \Sset} q^{\top} d \leq \max_{d\in \Sset \cup \{v\}} q^\top d.$$
So it follows that $\max_{d\in \Sset \cup \{v\}} u^{*\top}d \leq \max_{d \in S \cup \{v\}} q^\top d$ where $q$ was an arbitrary unit vector. Hence $u^* \in \cv(\Sset \cup \{v\})$ and $\cm(\Sset \cup \{v\}) = \cm(\Sset)$. 
\end{proof}


        
        
            


\subsection{Minimal Positive Bases} \label{sec:minimal-delta-angle-pb}
Here we introduce the construction of two minimal positive bases, along with the corresponding cosine measures. The first is the canonical minimal positive basis, which was defined in \cite{DFO-textbook-2009}. 
\begin{definition} (Canonical Minimal Positive Basis)
    The canonical minimal positive basis of dimension $n$ is defined to be $\Pset =\{e_1,e_2,...,e_n,-e\}$.
\end{definition}
\begin{lemma}\cite[Lemma 3.1]{hare-k-spanning-2024}
    The cosine measure of the canonical minimal positive basis of dimension $n$ is $\frac{1}{\sqrt{n^2+2(n-1)\sqrt{n}}}$.
\end{lemma}

{The second minimal positive basis is the minimal $\delta$-shift positive basis. This construction produces a family of positive bases with varying cosine measure using the parameter $\delta$. When $\delta = 0$, we may recover the canonical uniform minimal positive basis, which we define as follows.
\begin{definition}
The \textit{canonical uniform simplex} $P$ is the matrix
    \begin{align*}
    P = \begin{bmatrix}
            a_1 & -\frac{a_1}{n} & -\frac{a_1}{n} & \cdots & -\frac{a_1}{n} & -\frac{a_1}{n} & \vline & -\frac{a_1}{n} \\
            0 & a_2 & -\frac{a_2}{n-1} & \cdots & -\frac{a_2}{n-1} & -\frac{a_2}{n-1} & \vline & -\frac{a_2}{n-1} \\
            0 & 0 & a_3 & \cdots & -\frac{a_3}{n-2} & -\frac{a_3}{n-2} & \vline & -\frac{a_3}{n-2} \\
            \vdots & \vdots & \vdots & \ddots & \vdots & \vdots & \vline & \vdots \\
            0 & 0 & 0 & \cdots & a_{n-1} & -\frac{a_{n-1}}{2} & \vline & -\frac{a_{n-1}}{2} \\
            0 & 0 & 0 & \cdots & 0 & a_n & \vline & -a_n
        \end{bmatrix}
\end{align*}
and $a_i = \sqrt{\frac{(n-i+1)(n+1)}{n(n-i+2)}},~i = 1, 2, \dots, n.$ Let $\Pset$ be a set containing only the columns of $P$. We define $\Pset$ to be the \emph{canonical uniform minimal positive basis}. 
\end{definition}
It was verified in \cite[Theorem 16]{jarry-bolduc_uniform_2020} that the canonical uniform minimal positive basis is indeed a positive basis. Moreover, it is known to have the maximal cosine measure over all minimal positive bases \cite{naevdal_positive_2019}. 
}
\begin{definition}\cite[Definition 8]{jarry-bolduc_uniform_2020} (Uniform Minimal Positive Basis) The set $\Pset =\{d_1, d_2,...,d_{n+1}\}$, where $\|d_i\|=m > 0$ for $i = 1, 2,...,n+1$ is a uniform minimal positive basis for $\R^n$ if and only if $\Pset$ positively spans $\R^n$ and $d_i^\top d_j =\frac{-m^2}{n}$, $1 \leq i < j \leq n+1$.
\end{definition}
\begin{theorem}\cite[Theorem 1]{naevdal_positive_2019}\label{thm:optimal-max-cm-for-min-pbasis}
    Let $\Pset = \{d_1, d_2,...,d_{n+1}\}$ be a (minimal) positive basis of unit vectors for $\R^n$. Then the following holds:
    \begin{enumerate}
        \item $\cm(\Pset) \leq \frac{1}{n}$.
        \item $\cm(\Pset) = \frac{1}{n} \iff d_i^\top d_j = -\frac{1}{n}$, for all $i \not =j$.
    \end{enumerate}
\end{theorem}
Using Theorem~\ref{thm:n+1-arbitrary-set} and Corollary~\ref{corollary:min-pos-basis-delta-choice}, which are presented below, we may construct any positive basis of size $n+1$ with cosine measure in $(0, \frac{1}{n}]$. Thus our theorem enables us to construct a minimal positive basis for all possible cosine measures.

The intuition of the theorem is as follows. Starting from a uniform minimal positive basis, we note that the cosine vector is $-e_1$. We  shift the vectors $d_2, \dots, d_{n+1}$ by adding $\delta e_1$, which ``widens the gap'' and decreases the cosine measure. Here, the parameter $\delta$ controls how much we wish to ``widen the gap''. Figure~\ref{fig:n+1} illustrates the application of the theorem. The black vectors indicate the canonical uniform simplex, the dashed red line indicates the shift, and the resulting vectors after normalization are in blue. Notice after shifting the vector has length $\frac{1}{\alpha}$ and is no longer unit, hence we scale our vectors by $\alpha$ in the theorem. Notice that the green cosine vector is the same for both the original and new sets. As we shift, the angle between the set and the cosine vector increases, thus decreasing the cosine measure. 

\begin{figure}[ht]
  \begin{center}
    \includegraphics[width=0.6\textwidth]{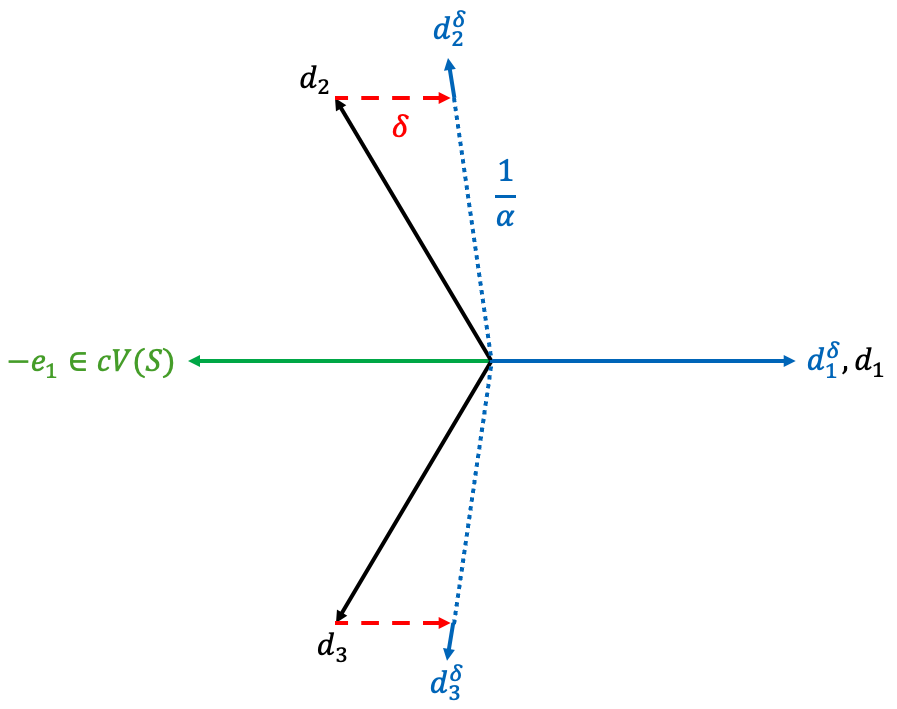}
    \caption{\label{fig:n+1}An example of Theorem~\ref{thm:n+1-arbitrary-set} in $\R^2$. }
  \end{center}
\end{figure}

The following theorem will allow us to construct sets of size $n+1$ with arbitrary cosine measure, for which we call the minimal $\delta$-shift positive basis. 

\begin{restatable}{theorem}{minpbasis} \label{thm:n+1-arbitrary-set}
    Let $n \geq 2$ and $0 \leq \delta < \frac{1}{n}$. Let $P^0 = [d_1 \dots d_{n+1}]$ denote the canonical uniform simplex. Consider the matrix $P^\delta$ defined by 
    \begin{align*}
        [P^\delta]_i = \begin{cases}
            d_i & \text{if } i = 1\\
            \alpha (d_i + \delta e_1) & \text{if } 2 \leq i \leq n+1
        \end{cases}
    \end{align*} where $\alpha = \frac{n}{\sqrt{n^2\delta^2 - 2 n\delta + n^2}}$ is a normalizing factor. Let $\Pset^\delta=\{d^\delta_1,\dots,d^\delta_{n+1}\}$ denote the set of columns of $P^\delta$. Then $\Pset^\delta$ is a positive basis with cosine measure $ \frac{1-\delta n}{\sqrt{n^2\delta^2 - 2 n\delta + n^2}}$.    
\end{restatable}

To prove this, we require the following lemmas. 
\begin{lemma}\label{lemma:min-pos-basis}
     Let $0 \leq \delta < \frac{1}{n}$, $P^0 = [d_1 \dots d_{n+1}]$, and $ P^\delta = [d^\delta_1 \ \cdots \ d^\delta_{n+1}]$ be as defined in Theorem~\ref{thm:n+1-arbitrary-set}. Define the quantities:
     \begin{align*}
         \alpha &= \frac{n}{\sqrt{n^2\delta^2 - 2 n\delta + n^2}}, \\
         \beta &= \alpha^2 \left(\frac{n\delta^2 - 2\delta -  1}{n}\right),\\
         \gamma &= -\alpha(1-\delta n)-(n-1).
     \end{align*}
     Then the following hold
     \begin{enumerate}[label={(\roman*)}]
         \item $\alpha = \frac{1}{\|d_i + \delta e_1\|}$ for all $i \in \{2, \dots, n+1\},$ \label{lemma:min-pos-basis:normalization}
          \item  $ \beta = {d^\delta_j}^\top {d^\delta_k}$ for all $2 \leq j,k \leq n+1$ and $j \not = k$, \label{lemma:min-pos-basis:beta-equal-dd}
         \item $\beta \leq -\frac{1}{n}$ and $\left(1+\frac{\beta n}{1-\beta}\right) > 0$, \label{lemma:min-pos-basis:beta-bound}
         \item  $-n \leq \gamma < 0$. \label{lemma:min-pos-basis:gamma-bound}
     \end{enumerate}
\end{lemma}

\begin{proof}~\ref{lemma:min-pos-basis:normalization}
    Let us first see that $\alpha$ is well defined. Note that
    \begin{align*}
        n^2\delta^2-2n\delta+n^2 
        &= n^2\left(\left(\delta - \frac{1}{n}\right)^2 + 1-\frac{1}{n^2}\right) \geq 0,
    \end{align*}
    where the last inequality holds since $n \geq 2$ so $1-\frac{1}{n^2} \geq 0$. By the definition of a uniform simplex we have $\|d_i\|=1$ for all $i\in\{1, \dots, n\}$. So
    \begin{align*}
        ([d_i]_1)^2 + \dots + ([d_i]_{n+1})^2 = 1. 
    \end{align*} Now, 
    \begin{align*}
        \frac{1}{\|d_i + \delta e_1\|} 
        &=\frac{1}{\sqrt{\delta^2 + 2\delta[d_i]_1 + 1}} 
        =\frac{1}{\sqrt{\delta^2 + 2\delta\left(-\frac{1}{n}\right) + 1}}         
         = \alpha.
    \end{align*}
\ref{lemma:min-pos-basis:beta-equal-dd} For any $2 \leq j,k \leq n+1$ where $j \not = k$ we have
\begin{align*}
    {d_j}^{\delta \top} d^\delta_k &= (\alpha(d_j + \delta e_1))^\top (\alpha(d_k + \delta e_1)) \\ &=\alpha^2({d_j}^\top d_k + (\delta e_1)^\top d_j + (\delta e_1)^\top d_k + (\delta e_1)^\top(\delta e_1)) \\
    &= \alpha^2 \left(-\frac{1}{n} -\frac{\delta}{n}-\frac{\delta}{n} + \delta^2 \right) =\beta,
\end{align*}
where we used the fact that $d_j^\top d_k = -\frac{1}{n}$ by Theorem~\ref{thm:optimal-max-cm-for-min-pbasis}. \\
\ref{lemma:min-pos-basis:beta-bound} We first show that $\beta$ is decreasing as a function of $\delta$. Recall that $d^\delta_{i} = \alpha(d_i + \delta e_1)$ and $\beta = d^{\delta\top}_i d^\delta_{j}$ for $2 \leq i,j \leq n+1$ where $i\not = j$. Let $0 \leq \delta_1 < \delta_2 < \frac{1}{n}$. Then
    \begin{align*}
        d_i^{\delta_1\top} d_j^{\delta_1} > d_i^{\delta_2\top} d_j^{\delta_2}  
        &\iff  \alpha^2 \left(\frac{n\delta_1^2 - 2\delta_1 -  1}{n}\right) >  \alpha^2 \left(\frac{n\delta_2^2 - 2\delta_2 -  1}{n}\right) \\
        &\iff n \delta_1^2-2\delta_2 > n \delta_2^2 - 2\delta_2 \\
        &\iff n(\delta_1 - \delta_2)(\delta_1 + \delta_2) + 2(\delta_2-\delta _1) > 0 \\
        &\iff \delta_1 + \delta_2 < \frac{2}{n}.
    \end{align*}
    By our assumption that $\delta_1 < \frac{1}{n}$ and $\delta_2 < \frac{1}{n}$, the last inequality holds. Now let us rewrite $\beta$ only in terms of $\delta$ and $n$ as
    \begin{align*}
        \beta &= \alpha^2 \left(\frac{n\delta^2 - 2\delta -  1}{n}\right) = \frac{n(n^2 \delta^2-2\delta - 1)}{\delta^2 n^2 - 2\delta n + 1 + n^2 - 1} = \frac{n \delta^2 - 2\delta -1}{n \delta^2 - 2\delta + n}.
    \end{align*}
    Since $\beta$ is decreasing, if $0 \leq \delta < \frac{1}{n}$, we know 
    $\beta \leq \frac{n \left(0\right)^2 - 2\left(0\right) -1}{n \left(0\right)^2 - 2\left(0\right) + n} = -\frac{1}{n}$.
    To show the second part of the statement, note that
    \begin{align*}
        \left(1 + \frac{\beta n}{1- \beta}\right) > 0 
        &\iff \beta > -\frac{1}{n-1}.
    \end{align*}
    Since $\beta$ is decreasing, if $0 \leq \delta < \frac{1}{n}$, we know
    $
        \beta > \frac{n \left(\frac{1}{n}\right)^2 - 2\left(\frac{1}{n}\right) -1}{n \left(\frac{1}{n}\right)^2 - 2\left(\frac{1}{n}\right) + n} = -\frac{1}{n-1}.$
\ref{lemma:min-pos-basis:gamma-bound} We first note that $-\alpha(1-\delta n)$ is an increasing function in $\delta$ since the derivative is non-negative. 
    Then the value of $-\alpha(1-\delta n)$ is bounded by
    \begin{align*}
        -\alpha(1-\delta n) = \frac{-n(1-\delta n)}{\sqrt{\delta^2n^2 - 2\delta n + n^2}} < \frac{-n\left(1-\left(\frac{1}{n}\right) n\right)}{\sqrt{\left(\frac{1}{n}\right)^2n^2 - 2\left(\frac{1}{n}\right) n + n^2}} = 0,
    \end{align*}
    so we have
    \begin{align*}
         \gamma = -\alpha(1-\delta n)-(n-1) < - (n-1) < 0.
    \end{align*}
    Similarly,
    \begin{align*}
        -\alpha (1-\delta n) = \frac{-n(1-\delta n)}{\sqrt{\delta^2n^2 - 2\delta n + n^2}} \geq \frac{-n(1-\left(0\right) n)}{\sqrt{\left(0\right)^2n^2 - 2\left(0\right) n + n^2}} = -1,
    \end{align*}
    so
    \begin{align*}
        \gamma = -\alpha(1-\delta n)-(n-1) \geq -1 - (n-1) = -n.
    \end{align*}
\end{proof}

\begin{lemma} \label{lemma:min-pbasis-gram}
    Let $0 \leq \delta < \frac{1}{n}$. Let 
    \begin{align*}
        P^\delta = [d^\delta_1 \ \cdots \ d^\delta_{n+1}] \text{ and }
        P^0 = \begin{bmatrix}d_1 \ \cdots \ d_{n+1}\end{bmatrix}
    \end{align*}
    be as defined in Theorem~\ref{thm:n+1-arbitrary-set}, and $\alpha, \beta, \gamma$ be as defined in Lemma~\ref{lemma:min-pos-basis}.
    Let $\B = \{d^\delta_2, \dots, d^\delta_{n+1}\}$ and $\T = \{d^\delta_1, \dots,d^\delta_{n+1}\} \setminus \{d^\delta_i\}$ for some $i \in \{2, \dots, n+1\}$.
    Let $B$ and $T$ be the matrices constructed using the vectors in $\B$ and $\T$ as columns, respectively. Then,
    \begin{enumerate}[label={(\roman*)}]
        \item for any choice of $i \in \{2, \dots, n+1\}$ in the construction of $\T$, we have 
        $$\gram(T) = \begin{bmatrix}
            1 & \lambda & \cdots & \lambda \\
            \lambda & 1 & \cdots & \beta \\
            \vdots & \vdots & \ddots & \vdots \\ 
            \lambda & \beta & \cdots & 1
            \end{bmatrix},$$
            where $\lambda = \alpha\left(-\frac{1}{n} + \delta \right)$, \label{lemma:min-pbasis-gram:invariant}
        \item $\gram^{-1}(B) = \left(\frac{1}{1-\beta}\right)I - \left(\frac{\beta}{(1-\beta)^2\left(1 + \frac{\beta n}{1-\beta}\right)}\right)e^{n \times n},$ \label{lemma:min-pbasis-gram:B-inv}
        \item $\gram^{-1}(T) = M^{-1} \gram^{-1}(B) M^{-\top}$ where $$ M^{-1}=\begin{bmatrix}
            -\alpha(1-\delta n) & \mathbb{0}^{1 \times (n-1)} \\
            -e^{(n-1) \times 1} & I^{(n-1) \times (n-1)}
        \end{bmatrix}.$$ \label{lemma:min-pbasis-gram:S-inv}
    \end{enumerate}
\end{lemma}

\begin{proof}~\ref{lemma:min-pbasis-gram:invariant}
     First note that for any $2 \leq j,k \leq n+1$ where $j \not = k$ we have
    \begin{align*}
        d_1^{\delta\top} d^{\delta}_j = d_1^\top (\alpha(d_j+\delta e_1)) = a_1 \left(\alpha \left(-\frac{a_1}{n}+\delta\right)\right) = \alpha \left(-\frac{1}{n} + \delta\right).
    \end{align*}
    For convenience, define $\lambda = \alpha\left(-\frac{1}{n} + \delta\right)$. Recall from Lemma~\ref{lemma:min-pos-basis}~\ref{lemma:min-pos-basis:beta-equal-dd} that ${d_j^\delta}^\top d_k^\delta = \beta$ for all $2 \leq j,k \leq n+1$ where $j \not = k$. 
    So for any choice of $i \in \{2, \dots, n+1\}$ we have 
    \begin{align*}
        \gram(T) = T^\top T = \begin{bmatrix}
        1 & \lambda & \cdots & \lambda \\
        \lambda & 1 & \cdots & \beta \\
        \vdots & \vdots & \ddots & \vdots \\ 
        \lambda & \beta & \cdots & 1
        \end{bmatrix}.
\end{align*}
\ref{lemma:min-pbasis-gram:B-inv} We have
    \begin{align*}
        \gram(B) = B^\top B = \begin{bmatrix}
        1 & \beta & \cdots & \beta \\
        \beta & 1 & \cdots & \beta \\
        \vdots & \vdots & \ddots & \vdots \\ 
        \beta & \beta & \cdots & 1
        \end{bmatrix} = (1-\beta)I + \beta e^{n \times n},
    \end{align*}
    and via the rank-1 update formula
    \begin{align*}
        \gram(B)^{-1} = \left(\frac{1}{1-\beta}\right)I - \left(\frac{\beta}{(1-\beta)^2\left(1 + \frac{\beta n}{1-\beta}\right)}\right)e^{n \times n}.
    \end{align*}
    
   ~\ref{lemma:min-pbasis-gram:S-inv}
    By~\ref{lemma:min-pbasis-gram:invariant}, we may consider any $i \in \{2, \dots, n+1\}$ in the construction of $\Sset$. Without loss of generality, fix $i = 2$ so $\T = \{d^\delta_1, \dots, d^\delta_{n+1}\} \setminus \{d^\delta_2\}$. First note that 
    \begin{align*}
        d^\delta_1 = e_1 = -\frac{1}{\alpha(1-\delta n)} \sum_{2 \leq j \leq n+1} d^\delta_j,
    \end{align*}
    so we may write
    \begin{align*}
        T =  \begin{bmatrix}
            e_1 & d^\delta_3 & \cdots & d^\delta_{n+1}
        \end{bmatrix} = B \begin{bmatrix}
            -\frac{1}{\alpha(1-\delta n)} & \mathbb{0}^{1 \times (n-1)} \\
            -\frac{1}{\alpha(1-\delta n)} e^{(n-1) \times 1} & I^{(n-1) \times (n-1)}
        \end{bmatrix}.
    \end{align*}
    Thus 
        $$\gram^{-1}(T) = ((BM)^\top BM)^{-1}
        = M^{-1}(B^\top B)^{-1} M^{-\top}= M^{-1} \gram^{-1}(B) M^{-\top}$$
    where, by Gaussian elimination, we can easily find 
    \begin{align*}
        M^{-1}=\begin{bmatrix}
            -\alpha(1-\delta n) & \mathbb{0}^{1 \times (n-1)} \\
            -e^{(n-1) \times 1} & I^{(n-1) \times (n-1)}
        \end{bmatrix}.
    \end{align*}
\end{proof}

We are now ready to prove Theorem~\ref{thm:n+1-arbitrary-set}.

\begin{proof} (Theorem~\ref{thm:n+1-arbitrary-set})
    Let us first show $\Pset^\delta$ is positive spanning. By Theorem~\ref{thm:pos-spanning-sum-equals-0} it suffices to show $(1-\delta n)d^\delta_1 + \sum_{i=2}^{n+1}\left(\frac{1}{\alpha}\right)d^\delta_i = \0$ since $1-\delta n >0$ and $\frac{1}{\alpha} > 0$. Note 
    \begin{align*}
         (1-\delta n)d^\delta_1 &+ \sum_{i=2}^{n+1}\left(\frac{1}{\alpha}\right)d^\delta_i = \left(1-\delta n\right)d^\delta_1 + \left(\frac{1}{\alpha}\right)d^\delta_2 + \dots+  \left(\frac{1}{\alpha}\right)d^\delta_{n+1}\\
         &= \left(1-\delta n\right)d_1 + \left(\frac{1}{\alpha}\right)\alpha(d_2+\delta e_1) + \dots+  \left(\frac{1}{\alpha}\right)\alpha(d_{n+1}+\delta e_1)\\
         &= -\delta n d_1 + \delta n e_1 + (d_1 + \dots d_{{n+1}}) \\
         &= -\delta n d_1 + \delta n e_1 \quad \because {\sum_{i=1}^{n+1} d_i = 0}\\
         &= 0.
    \end{align*}
    {Since $\Pset^\delta$ contains $n+1$ vectors, no proper subset of $\Pset^\delta$ positively spans $\R^n$, so $\Pset^\delta$ is a positive basis.}
    
    Let $\B = \{d^\delta_2, \dots, d^\delta_{n+1}\}$ and $g = -e_1$. By Lemma~\ref{lemma:min-pos-basis}~\ref{lemma:min-pos-basis:normalization} we know the columns of $\B$ are unit vectors. Applying Lemma~\ref{theorem:all-inner-prod-equal-implies-gram} and noting
    \begin{align} \label{eqn:gram-value}
        g^\top{d^\delta_i} = -e_1^\top d^\delta_i = -\alpha\left(-\frac{a_1}{n} + \delta \right) = \alpha\left(\frac{1}{n} - \delta \right) > 0,
    \end{align}
    for all $2 \leq i \leq n+1$, we see that $g$ is the Gram vector of $\B$. Additionally, $g^\top d^\delta_1 = -1(a_1) = -1$, so 
    \begin{align*} 
        \cm(\Pset^\delta, g) = \max_{d \in \Pset^\delta} d^\top g = \max\left\{-1, \alpha\left(\frac{1}{n} - \delta \right)\right\} = \alpha\left(\frac{1}{n} - \delta \right).
    \end{align*}
    We show $g$ is optimal by comparing the Gram value of $\B$ with the Gram value of all other possible sub-bases of $\Pset^\delta$, which is sufficient by Corollary~\ref{cor:necessary-cond-gram-vector}. Let $\T \subset \Pset^\delta$ be arbitrarily chosen. Since the columns in $B$ and $T$ are unit vectors by Lemma~\ref{lemma:min-pos-basis}~\ref{lemma:min-pos-basis:normalization}, we may define $\gram_1 = \gram(B)$ and $\gram_2 = \gram(T)$. In particular we wish to show 
   \begin{equation}\frac{1}{\sqrt{e^\top \gram_1^{-1}}e} \leq \frac{1}{\sqrt{e^\top \gram_2^{-1}}e}.\label{eqn:mpb-main-inequality}\end{equation}
    If $\T \not = \B$, we may reorder the columns of $T$ such that $d^\delta_1$ is the first column. By Lemma~\ref{lemma:min-pbasis-gram}~\ref{lemma:min-pbasis-gram:invariant}, for any choice of $i \in \{2, \dots, n+1\}$ the set $\Sset$ yields the same Gram matrix, so without loss of generality we can fix $i=2$ so $\T = \{d^\delta_1, \dots, d^\delta_{n+1}\} \setminus \{d^\delta_2\}$.  
    Then, by Lemma~\ref{lemma:min-pbasis-gram}~\ref{lemma:min-pbasis-gram:S-inv} we may write 
    $\gram_2 = M^{-1} \gram^{-1}_1 M^{-\top}$ where 
    \begin{align*}
        M^{-1} = \begin{bmatrix}
            -\alpha(1-\delta n) & \mathbb{0}^{1 \times (n-1)} \\
            -e^{(n-1) \times 1} & I^{(n-1) \times (n-1)}
        \end{bmatrix}.
    \end{align*}
    For convenience let $$z^\top = e^\top M^{-1} = \begin{bmatrix}
        \gamma & 1 & \cdots & 1
    \end{bmatrix},$$ where $\gamma = -\alpha(1-\delta n) - (n-1)$. We now prove inequality \eqref{eqn:mpb-main-inequality}, as
    \begin{align*}
        &\frac{1}{\sqrt{e^\top \gram_1^{-1} e}} \leq \frac{1}{\sqrt{e^\top \gram^{-1}_2 e}} \\
        &\iff e^\top \gram_1^{-1} e \geq e^\top \gram_2^{-1} e \\
        &\iff e^\top \gram_1^{-1} e \geq z^\top \gram_1^{-1} z \\
        &\iff \sum_{i,j} [\gram_1^{-1}]_{i,j} - \sum_{2 \leq i,j \leq n} [\gram_1^{-1}]_{i,j} \\
        & \qquad \quad \geq \gamma \sum_{2 \leq i \leq n} [\gram_1^{-1}]_{i,1} + \gamma \sum_{2 \leq j \leq n} [\gram_1^{-1}]_{1,j} + \gamma^2 [\gram_1^{-1}]_{1,1} \\
        &\iff \sum_{2 \leq i \leq n} [\gram_1^{-1}]_{i,1} + \sum_{2 \leq j \leq n} [\gram_1^{-1}]_{1,j} +  [\gram_1^{-1}]_{1,1} \\
        &\qquad \quad \geq  2(n-1)\gamma [\gram_1^{-1}]_{2,1}  + \gamma^2 [\gram_1^{-1}]_{1,1} \\
        &\iff 2(n-1)(1-\gamma) [\gram_1^{-1}]_{2,1}  + (1-\gamma^2) [\gram_1^{-1}]_{1,1} \geq 0 \\
        &\iff 2(n-1)(1-\gamma)\left(\frac{-\beta}{(1-\beta)^2\left(1 + \frac{\beta n}{1-\beta}\right)}\right) \\
        & \qquad + (1-\gamma^2) \left(\frac{1}{1-\beta} - \frac{\beta}{(1-\beta)^2\left(1 + \frac{\beta n}{1-\beta}\right)}\right)\geq 0 \quad \text{by Lemma~\ref{lemma:min-pbasis-gram}~\ref{lemma:min-pbasis-gram:B-inv}} \\
        &\iff 1-\beta n + \gamma(1+(n-2)\beta) \geq 0.
    \end{align*}
    Let us inspect the last inequality. We know from Lemma~\ref{lemma:min-pos-basis}~\ref{lemma:min-pos-basis:beta-bound} that $\beta \leq - \frac{1}{n}$ so it follows that $1- \beta n \geq 2$. In addition
    \begin{align*}
        \beta \leq -\frac{1}{n} \implies 1 + (n-2) \beta \leq \frac{2}{n} \implies \gamma(1 + (n-2)\beta) \geq \frac{2\gamma}{n} \geq -2
    \end{align*}
    where the last inequality holds since $\gamma \geq -n$ by Lemma~\ref{lemma:min-pos-basis}~\ref{lemma:min-pos-basis:gamma-bound}. Adding $1-\beta n \geq 2$ and $\gamma(1 + (n-2)\beta) \geq -2$ yields the desired result.
\end{proof}
\begin{corollary}\label{corollary:min-pos-basis-delta-choice}
    Let $n \geq 2$. For any $0 < c \leq \frac{1}{n}$ we can construct a positive basis $P$ such that $\cm(P) = c$ by setting $\delta = \frac{1}{n} + \frac{\sqrt{-(n^2-1)(c^4 - c^2)}}{n(c^2-1)}$ in the construction above.
\end{corollary}

\subsection{Maximal Positive Basis}\label{sec:maximal-delta-angle-pb}
Here we introduce the construction of a maximal positive basis, along with the corresponding cosine measure. The maximal canonical positive basis can be traced back to \cite{davis-positive-independence}, and in \cite{naevdal_positive_2019} it was shown to have maximal cosine measure, that is for all positive bases of size $2n$, the canonical maximal positive basis yields the largest cosine measure. 

\begin{definition} (Canonical Maximal Positive Basis)
    The canonical maximal positive basis of dimension $n$ is defined to be $$\Pset =\{e_1,e_2,\dots,e_n,-e_1,\dots,-e_n\}.$$
\end{definition}

\begin{theorem}\cite[Theorem 2]{naevdal_positive_2019}\label{thm:max-cm-for-max-pbasis}
    Let $$\Pset = \{d_1, d_2,...,d_n, -\lambda_1d_1, -\lambda_2d_2,...,-\lambda_nd_n \}$$ where $\lambda_i > 0$ for $1 \leq i \leq n$ be a (maximal) positive basis for $\R^n$. Then the following holds: 
    \begin{enumerate}
        \item $\cm(\Pset) \leq \frac{1}{\sqrt n}$,
        \item $\cm(\Pset) = \frac{1}{\sqrt n} \iff d^\top_i d_j = 0$ for $i\not= j$ and $1 \leq i, j \leq n$.
    \end{enumerate}
\end{theorem}

As with the previous section we seek to create a maximal positive basis with varying cosine measure. Using Theorem~\ref{thm:test-set-2n} and Corollary~\ref{corollary:2n-delta-choice}, we may construct any positive basis of size $2n$ with cosine measure in $(0,\frac{1}{\sqrt n}]$. 
Thus our theorem enables us to construct a positive basis for all possible cosine measures in that range.

The intuition of the theorem follows similarly to the previous section. Starting with the canonical maximal positive basis, we note that the cosine vector is $\frac{e}{\|e\|}$ and shift the vectors in the direction $e$ and $-e$ to ``widen the gap'', decreasing the cosine measure. 


The following theorem will allow us to construct sets of size $2n$ with arbitrary cosine measure. 
The subsequent corollary to the presented theorem provides a concrete formula for choosing the necessary $\delta$ to achieve a particular cosine measure. 

\begin{theorem}\label{thm:test-set-2n}
    Let $n \geq 2$ and $0 \leq \delta < \frac{1}{n}$. Consider the matrix $B = I - \delta e^{n \times n}$ and the set of vectors $\B = \{b_1,b_2,\dots, b_n\}$ defined by the columns of $B$. Define $\alpha = \frac{1}{\sqrt{\delta^2 n - 2\delta + 1}}$ and $\Pset^\delta = \alpha \B \cup (-\alpha \B)$. Then $\Pset^\delta$ is a positive basis with cosine measure $\frac{1-\delta n}{\sqrt{n(\delta^2 n - 2\delta + 1)}}$. 
\end{theorem}
\begin{proof}
    The proof follows a similar process to the proof of Theorem~\ref{thm:n+1-arbitrary-set}. 
\end{proof}
\begin{corollary} \label{corollary:2n-delta-choice}
    Let $n \geq 2$. For any $0 < c \leq \frac{1}{\sqrt n}$ we can construct a positive basis $P$ such that $\cm(\Pset) = c$ by setting $\delta = \frac{1}{n} + \frac{\sqrt{(n-1)(c^2 - c^4 )}}{n(c^2-1)}$ in the construction above.
\end{corollary}

\subsection{Intermediate Positive Bases}\label{sec:intermediate-pbasis}
In \cite{hare_nicely_2023}, it was shown how to make an {\em optimal orthongonally structured} positive basis of any size between $n+1$ and $2n$.  

\begin{lemma}\cite[Corollary 34]{hare_nicely_2023}
    Let $\Pset_{n,s}$ be an optimal orthogonal positive basis. Let $\texttt{r} = \rem\left(\frac{n}{s-n}\right)$. Then 
    \begin{align*}
        \cm(\Pset_{n,s}) =  \frac{1}{\sqrt{(s-n-\texttt{r})\lfloor{\frac{n}{s-n}}\rfloor^2 + \texttt{r}\lceil \frac{n}{s-n}\rceil^2}}
    \end{align*}
\end{lemma}

It should be noted that the optimal orthogonally structured positive basis results in very symmetric structure. As such every Gram vector generated by a subbasis of the optimal orthogonally structured positive basis is optimal. We refer the reader to \cite{hare_nicely_2023} for more details. 


\subsection{Random Positive Spanning Set}\label{sec:random-pss}
To further increase the complexity of the test set, we generate a random positive spanning set. Although the cosine measure is unknown for this set, it may be useful to see how well an algorithm performs on a highly asymmetric and unstructured set. We present a theorem which will help us in the construction. 
\begin{theorem}\cite[Theorem 5.1]{regis_properties_2016}(Construction of positive spanning sets from a basis) Let $\B=\{v_1, \dots, v_k\}$ be a basis of $\R^n$ and let $\mathcal {J}=\{J_1,\dots, J_\ell\}$ be a collection of subsets of $K=\{1, \dots, k\}$ such that $\cup_{i=1}^\ell J_i=K$. Then the set 
    $$\B \cup \left\{- \sum_{j \in J_1} \lambda_{1,j}v_j, \dots, -\sum_{j \in J_l}{\lambda_{l,j}}v_j\right\}$$
where $\lambda_{i,j} > 0$ for all $i = 1, \dots, \ell$, $j \in J_i$ is positive spanning. 
\end{theorem}

In the context of the previous theorem, in our algorithm to generate a random positive spanning set, we randomize the collection $\mathcal J$, randomize the construction of the basis $\B$, and set $\lambda_{i,j} = 1$ for all $i,j$. The algorithm first generates a matrix with entries distributed uniformly. Summing over the entries in the columns and replacing the diagonal, we create a strictly diagonally dominant matrix, which is invertible and hence the columns form a basis. Then we randomly select columns in $B$ to create new vectors. Once all columns are considered, the algorithm terminates. 



        

        
        
            

            



\section{Numerical Results}\label{sec:numerics}
This section presents the methodology used to benchmark the algorithms introduced in Section~\ref{sec:algos} and the numerical results. Section~\ref{sec:setup} introduces how the test sets were stored, the parameters considered, and the number of tests conducted for each set type. In Section~\ref{sec:results} we discuss our numerical findings. 

\subsection{Setup} \label{sec:setup}
Table~\ref{table:set-summary-v2} outlines all the sets considered in our testing and the number of vectors in each set. Note that we add an augmented version of the maximum $\delta$-shift positive basis, adding an additional $n^2$ number of random vectors to incorporate a test with a large number of vectors.

Before benchmarking, all test cases were randomly rotated and permuted via a uniformly sampled rotation matrix built using the \texttt{special\_ortho\_group} function in \texttt{scipy}. This is then followed by a random reordering of the elements in the set. By Theorem~\ref{thm:rotation-invariant}, this rotation followed by permutation does not change the cosine measure of the set. 

\begin{table}[tbph] 
    \centering
    \caption{A summary of the types of positive bases (PB) and positive spanning sets (PSS) considered in our testing.
    \label{table:set-summary-v2}}
    {
    \begin{tabular}{lll} \toprule 
        Set Type & Cosine Measure & Number of Vectors 
        \\ \midrule 
        Min. Can. PB &  $\frac{1}{\sqrt{n^2+2(n-1)\sqrt{n}}}$ & $s=n+1$  \\
        Min. $\delta$-shift PB  &  $\frac{1-\delta n}{\sqrt{n(\delta^2n - 2 \delta + n)}}$ & $s=n+1$  \\ 
        Max. $\delta$-shift PB & $\frac{1-\delta n}{\sqrt{n(\delta^2 n - 2\delta + 1)}}$ & $s=2n$  \\ 
        Max. Aug. $\delta$-shift PB & $\frac{1-\delta n}{\sqrt{n(\delta^2 n - 2\delta + 1)}}$ & $s=2n + n^2$  \\ 
        Opt. Ortho. PB &  $\frac{1}{\sqrt{(s-n-\texttt{r})\lfloor{\frac{n}{s-n}}\rfloor^2 + \texttt{r}\lceil \frac{n}{s-n}\rceil^2}}$ & $s=\lfloor 1.25n \rfloor, \lfloor 1.75n\rfloor$  \\ 
        Random PSS &  Unknown & $n+1 < s < 2n$  \\ \bottomrule 
    \end{tabular}
    }
\end{table}

For each set type, we consider dimensions $n=$ 10, 13, 15, 18, 21, 24, 27, 30, 40, 50, 60, 70, 80, 90, and 100. The sets are stored as JSON files, using the keys \texttt{matrix} and \texttt{solution}. The former stores the set of vectors as a matrix, where the columns form the vector set. The latter stores a float which is the cosine measure of the set, if known. In general, the directory structure has the form \texttt{Set Type >> Dimension >> Modifier}. Modifiers can include varying parameter values of $\delta$ and set size $s$. We note that the JSON files in the directory contain only the original, non-transformed sets, constructed as described in Section~\ref{sec:testset}. The random rotation and permutation are applied during the benchmarking process and are not reflected in the stored test sets.

For the maximal and minimal $\delta$-shift positive bases, we consider the three parameters $\delta = 0, \frac{1}{2n},$ and $\frac{2}{3n}$. 
The augmented maximum $\delta$-shift positive basis and the random positive spanning set have random components, so three random instances are generated and benchmarked. All other sets are fixed, so only one instance is required. Then, for each instance and parameter, we test three different random rotations and permutations applied on each instance. We repeat this testing for each dimension. Table~\ref{table:set-instances} summarizes the number of tests in our benchmark. For example, for the optimal orthogonal positive basis, we consider 1 instance $\times$ 2 parameters $\times$ $3$ rotations and permutations $\times$ 15 dimensions for a total of $90$ test cases. 

\begin{table}[tbph]
    \centering
    \caption{A summary of the number of tests conducted for each set type.
    \label{table:set-instances}}
    \begin{tabular}{lccc}
    \toprule
    Set Type & Instances & Parameters & Number of Tests \\
    \midrule
    Min. Can. PB & 1 & - & 45 \\
    Min. $\delta$-shift PB & 1 & 3 & 135 \\
    Max. $\delta$-shift PB & 1 & 3 & 135 \\
    Aug. Max. $\delta$-shift PB & 3 & 3 & 405 \\
    Opt. Ortho. PB & 1 & 2 & 90 \\
    Random PSS & 3 & - & 135 \\
    \midrule
    Total & & & 945 \\
    \bottomrule
    \end{tabular}
\end{table}

We test all methods presented in Section~\ref{sec:algos}, running each method for 30 seconds on an Intel i5 10600K processor and 16GB of RAM. If the method is not completed, the best incumbent solution is returned. All methods were written in Python3. To test the vertex enumeration method, we used our own implementation of lexicographical reverse search proposed in \cite{Avis2000-LRS-lexicographical-reverse-search}. 
For the branch-and-bound method, we employed the QCLP and QP formulations of the problem and solved it using Gurobi v12.0.1 with default parameters. Additionally, we evaluated the multithreaded variants of both the Gurobi branch-and-bound solver and the random LP method, using 8 cores. For the random LP method and its multithreaded variant, we ran the algorithm four times and reported the average result. The full CSV of our results and implementations can be found on \href{https://github.com/ScholarSun/Cosine-Measure-Benchmarking}{Github}\footnote{https://github.com/ScholarSun/Cosine-Measure-Benchmarking}. 

\subsection{Results} \label{sec:results}

In this Section, we present accuracy profiles \cite{audet2017derivative} of our benchmarks, where the $x$-axis indicates the number of correct digits produced by the algorithm, and the $y$-axis shows the proportion of the tests considered solved to at least that level of relative accuracy. In summary, our numerical experiments suggest the following:
\begin{enumerate}
    \item non-enumerative methods outperform enumerative methods;
    \item branch-and-bound works better on sets with many vectors whereas random LP works better on smaller sets;
    \item minimal sets and optimal orthogonal sets are easy for enumerative methods, the random LP method, and the QP formulated branch-and-bound;
    \item performance of all algorithms degrade as dimension increases, but the non-enumerative methods scale better;
    \item when applying branch-and-bound, both QP and QCLP formulations perform similarly, except in a few cases.
\end{enumerate}
The items on the list are discussed in Section~\ref{subsec:fulldata}, ~\ref{subsec:large_vector},~\ref{subsec:minimal_sets_numerical},~\ref{subsec:dimensions}, and ~\ref{subsec:formulation} respectively. 

We note that for problems up to dimension 20, the branch-and-bound method consistently returned a solution with an optimality certificate within the 30-second time limit. However, in higher dimensions, no optimality guarantees were returned, effectively rendering this approach a heuristic. 
\subsubsection{Full Data Set}\label{subsec:fulldata}

Figure~\ref{fig:full_set} presents an accuracy profile of the methods across all test sets, excluding the random positive spanning set. Since the random positive spanning set lacks a known cosine measure, we analyze the results for that set separately. The accuracy profile in aggregate suggests that using Gurobi's branch-and-bound method works the best, followed by the random LP, vertex enumeration, basis enumeration, and finally the KKT enumeration. 

\begin{figure}[ht]
    \begin{center}
      \includegraphics[width=0.6\textwidth]{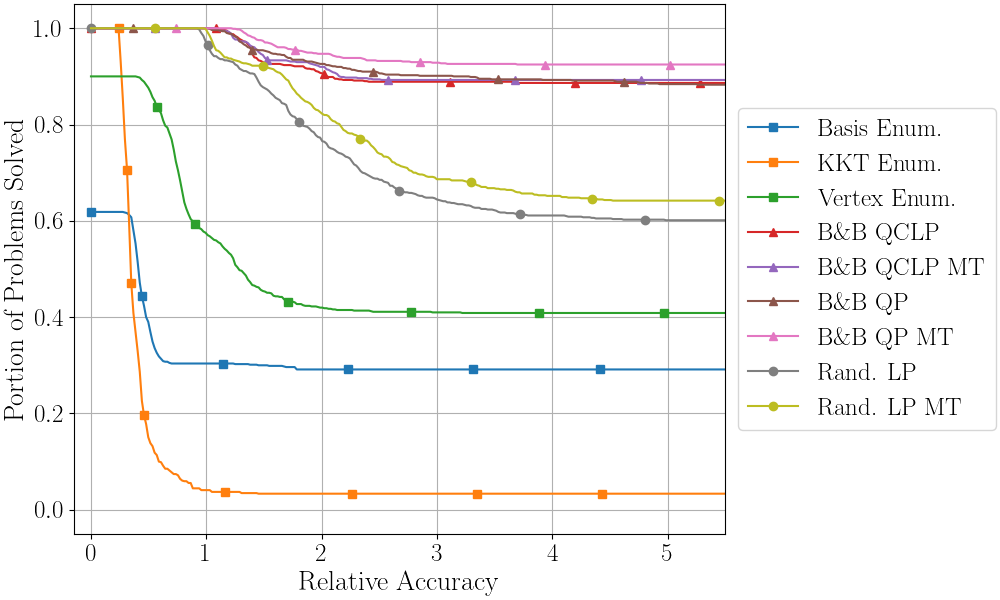}
      \caption{Accuracy profile of the entire test set, without random positive spanning sets. \label{fig:full_set}}
    \end{center}
\end{figure}

The poor performance of the KKT enumeration method is expected, as it  enumerates all subsets of size $2$ to $n$, which is substantially more than the basis enumeration method, which only considers subsets of size $n$. Similarly, the vertex enumeration method outperforms basis enumeration because it only enumerates subsets of size $n$ that correspond to vertices. This drastically reduces the number of subsets to enumerate. These observations are consistent across all our tests. 

Both the branch-and-bound and random LP methods perform substantially better than any enumerative technique, as they do not rely on exhaustively searching through all possible solutions. This trend appears across all tests, with only a few exceptions discussed in the later subsections. We also note that the random LP method and QP formulated branch-and-bound benefit significantly from multithreading, whereas the QCLP formulation sees little improvement with the addition of more cores.

\subsubsection{Impact of the Number of Vectors in the Set} \label{subsec:large_vector}

Figure~\ref{fig:noaug+aug}(a) presents the accuracy profile of the entire test set except the augmented $\delta$-shift maximal positive basis, whereas Figure~\ref{fig:noaug+aug}(b) presents the accuracy profile for the $\delta$-shift augmented maximal positive basis, which includes the tests for all choices of $\delta$. Surprisingly, all branch-and-bound methods achieve perfect performance, most likely due to the fact that adding more vectors decreases the sizes of all non-optimal gaps, allowing the algorithm to discard a large number of solutions easily. In comparison, sets with fewer vectors may contain more gaps of similar size to the optimal, making it difficult to reduce the feasible region effectively. 
\begin{figure}[ht]
    \begin{center}
        \includegraphics[width=\textwidth]{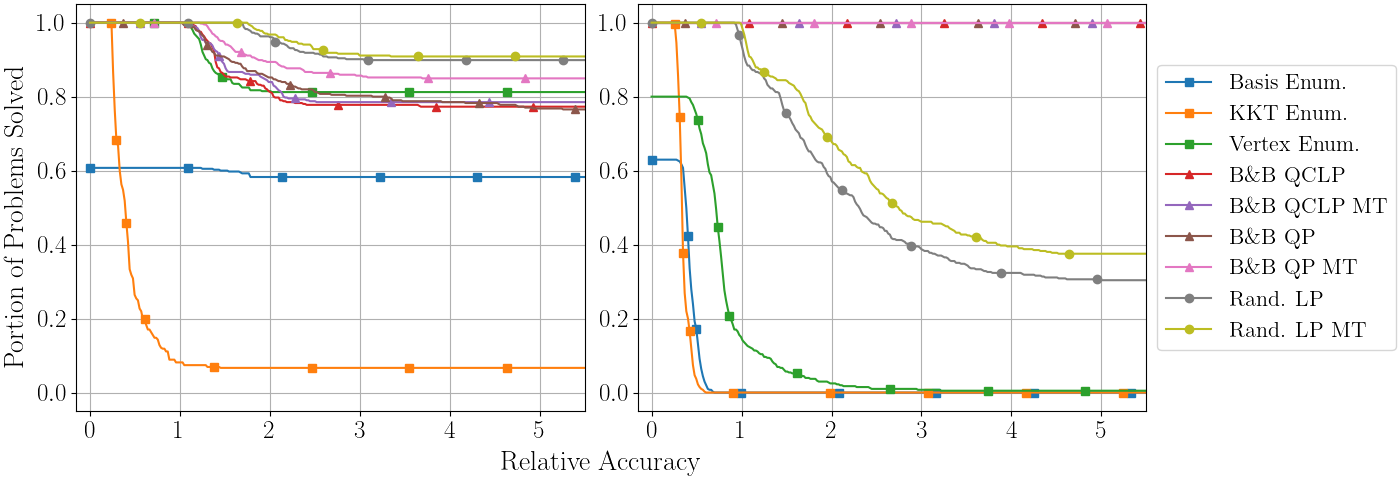}
        \caption{From left to right, the accuracy profile of the: (a) full test set without the augmented maximal $\delta$-shift positive basis and random positive spanning set, (b) augmented maximal $\delta$-shift positive basis. \label{fig:noaug+aug}}
    \end{center}
\end{figure}
Due to the large number of vectors in the set, the random LP method struggles as it often finds a solution corresponding to one of the many small, non-optimal gaps. Given that the augmented variant accounts for nearly half of our entire test set, this explains the large discrepancy between the random LP and branch-and-bound methods in the overall accuracy profile of Figure~\ref{fig:full_set}. 

Finally, we observe as expected, that increasing the $\delta$ parameter improves the performance of the random LP method, as it enlarges the size of the optimal ``gap''. This trend is illustrated in Figure~\ref{fig:augmented_delta}, in the Appendix, and confirms our intuition about the geometry of the problem.

With the augmented set removed, the accuracy profile in Figure~\ref{fig:noaug+aug}(a) suggests that the random LP method outperforms branch-and-bound. Aside from the augmented variant, all other sets in the test have less than or equal to $2n$ vectors, which suggests that the random LP method offers better performance when the set contains fewer vectors.  Notably, this trend also appears in the non-augmented maximal $\delta$-shift positive basis (which shares the same underlying structure as the augmented version) and in the highly random positive spanning set, as shown in Figures~\ref{fig:max_pbasis} and~\ref{fig:random_pspan}, respectively. This further suggests that the discrepancy in performance may be driven by the number of vectors in the set. 


\subsubsection{Impact of Set Structure} \label{subsec:minimal_sets_numerical}

Figure~\ref{fig:min+ortho}(a) presents the accuracy profile for all minimal sets, that is sets with exactly $n+1$ vectors. For these sets, we see that basis enumeration, vertex enumeration and random LP methods perform perfectly. This is expected, as the minimal sets contain only $n+1$ vectors, hence only $n$ bases need to be evaluated. Moreover, as a consequence of having just $n$ bases, the space is divided into $n$ distinct ``gaps'', making it very likely the random LP method will find a solution. In contrast, branch-and-bound is seen to perform worse in these settings. However, the QP formulation is only marginally worse while the QCLP formulation is significantly worse, which may be a side effect of the formulation.

Similarly, for the optimal orthogonal positive bases in Figure~\ref{fig:min+ortho}(b), both vertex enumeration and the random LP method achieve perfect performance, branch-and-bound performs slightly worse than perfect, and basis enumeration struggles. Since all vertices of the associated polytope have equal norm, this enables the vertex enumeration and random LP methods to instantly find the optimal solution. Basis enumeration, however, struggles because not all bases correspond to vertices.

\begin{figure}[ht]
    \begin{center}
      \includegraphics[width=\textwidth]{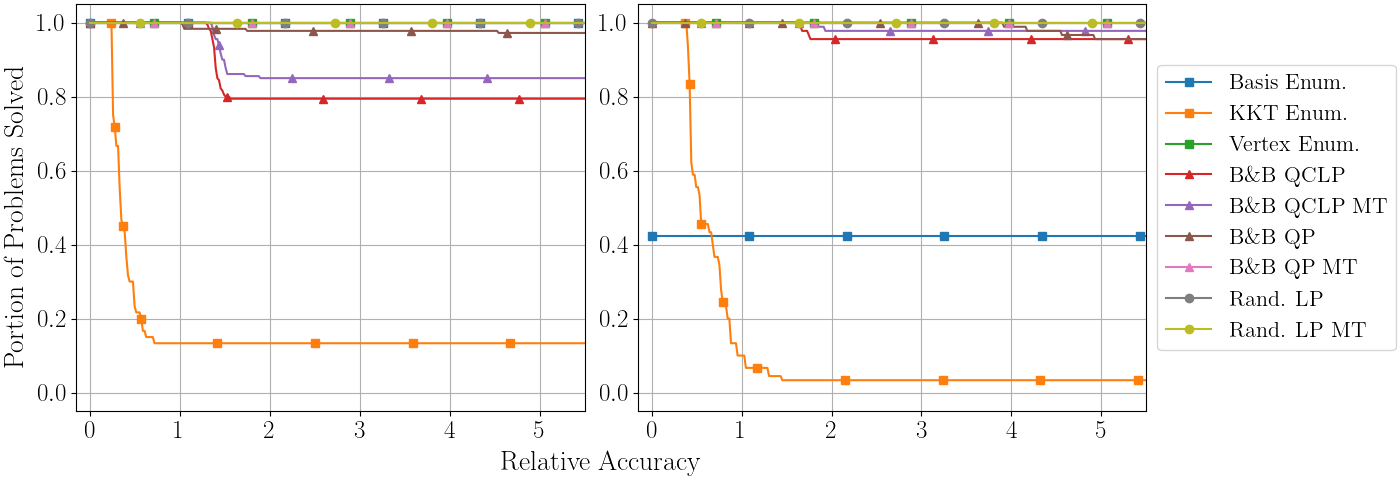}
      \caption{From left to right, the accuracy profile of: (a) all minimal sets, (b) the optimal orthogonal positive bases. \label{fig:min+ortho}}
    \end{center}
\end{figure}
\vspace{-1em}

\subsubsection{Impact of Dimension}\label{subsec:dimensions}
As expected, we observed that as the dimension increases, the performance of all algorithms deteriorates. Figure~\ref{fig:dimension_full}, in the Appendix, support this observation. Notably, the branch-and-bound curve does not decline as steeply as that of the random linear program; however, this is primarily due to the large number of augmented maximal $\delta$-shift positive bases in the test set, for which the branch-and-bound method performs perfectly across all dimensions. The augmented test set also contributes to the steep drop in the random LP method's performance, as the method tends to perform poorly on the augmented set described in Section~\ref{subsec:fulldata}.

\subsubsection{Impact of Formulation} \label{subsec:formulation}
When applying the branch-and-bound methods, it is worth noting that the solver performance is sensitive to the formulation used (QP or QCLP). Across our entire test set, we found in general both QP and QCLP formulations to perform similarly, with the QP formulation with multithreading slightly outperforming the rest. Both formulations solve the augmented set perfectly. 
It is not definitive if one formulation is universally better than the other as each performs better on different types of sets. That said, based on the overall accuracy profile across the entire test set, the QP formulation with multithreading may be marginally stronger.

\section{Conclusion}\label{sec:conclusion}
This paper explored solution methods for the NP-Hard cosine measure problem in higher dimensions. To tackle this, we proposed a novel reformulation of the cosine measure problem as norm maximization over a polytope and found that all solutions correspond to vertices. This reformulation inspired two new and simple algorithms, in particular the vertex enumeration method and the random LP method. 

To evaluate the performance of the algorithms, we developed a robust test set, including two new methods for the creation of minimal and maximal positive bases with known cosine measure. We then tested the efficacy of the new and existing algorithms using our test set.

In our numerical results, we found that the random LP method and branch-and-bound generally outperform the enumerative methods. Branch-and-bound performed better on sets with a larger number of vectors, while the random LP method was more effective for sets of size between $n+1$ and $2n$. Among the enumerative approaches, vertex enumeration performed the best, as it involved the fewest number of objects to enumerate.

There are several promising directions for future work. More extensive numerical testing could be done to help determine the threshold of the size of set at which to prefer the random LP method over branch-and-bound. Additionally, a tailor-made branch-and-bound algorithm leveraging the geometry of the reformulation may be possible. Lastly, for the random LP method, developing more sophisticated sampling methods could improve the performance of the algorithm.

\clearpage
\appendix
\section{Numerical Figures}
\label{sec:appendix-figures}

\begin{figure}[ht]
    \begin{center}
      \includegraphics[width=\textwidth]{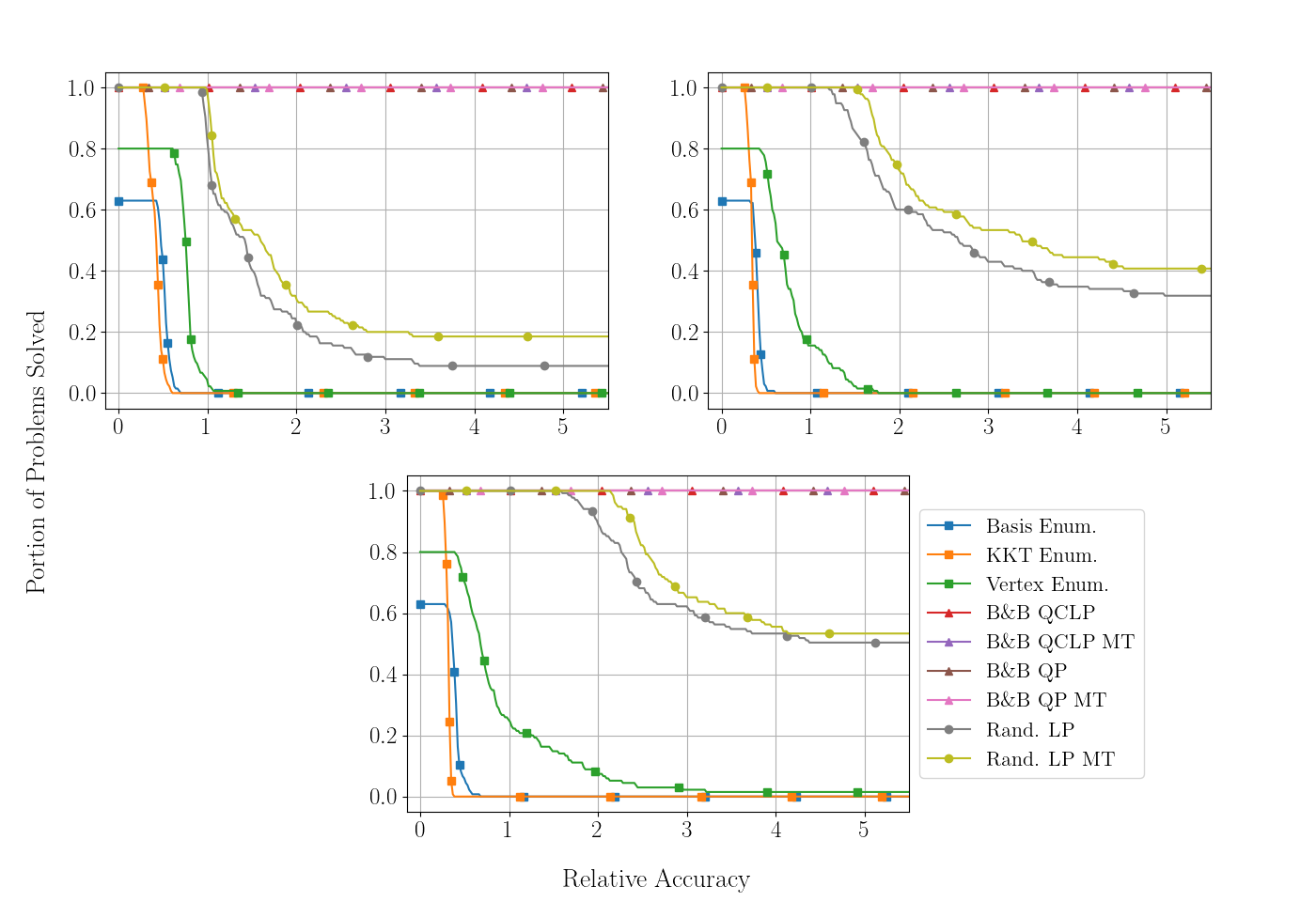}
      \caption{Clockwise, starting from the top left. The accuracy profile of the augmented maximal positive basis with $\delta = 0$, $\delta = \frac{1}{2n}$, and $\delta = \frac{2}{3n}$. \label{fig:augmented_delta}}
    \end{center}
\end{figure}

\begin{figure}[ht]
    \begin{center}
      \includegraphics[width=0.6\textwidth]{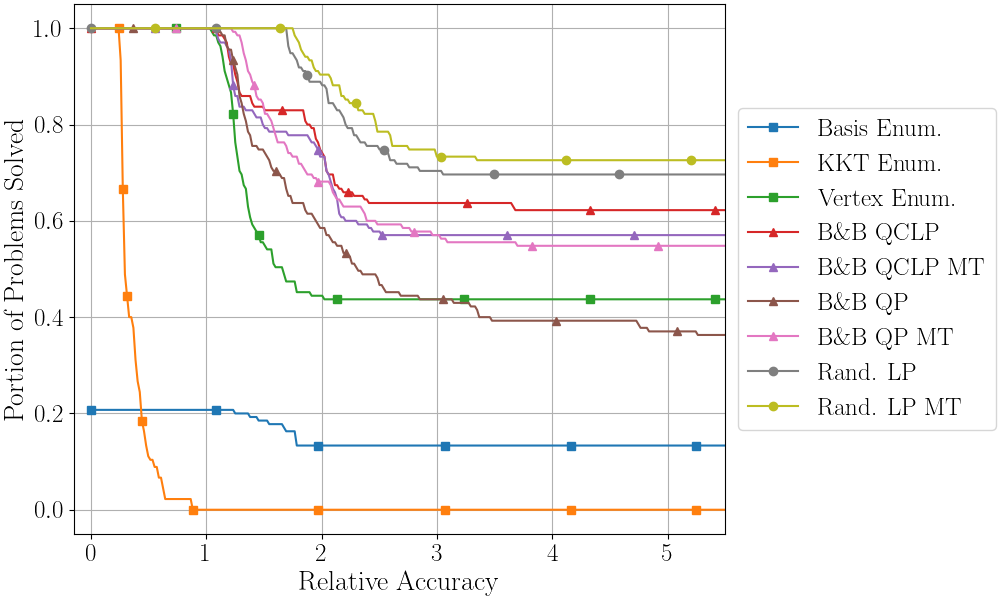}
      \caption{Accuracy profile of the maximal $\delta$-shift positive bases, with all $\delta$ parameters included. \label{fig:max_pbasis}}
    \end{center}
\end{figure}

\begin{figure}[ht]
    \begin{center}
      \includegraphics[width=0.6\textwidth]{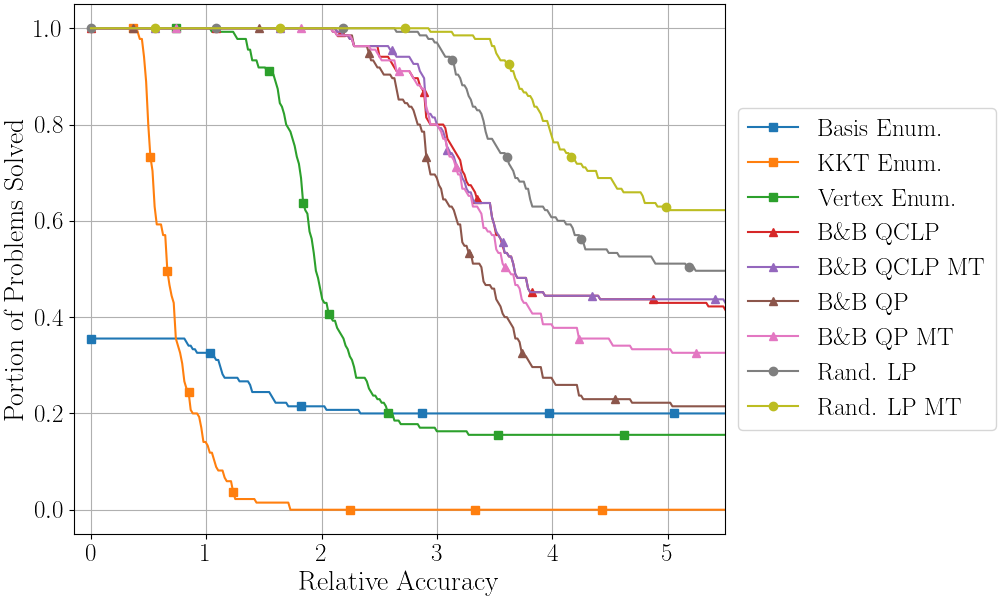}
      \caption{Accuracy profile of the random positive spanning sets. \label{fig:random_pspan}}
    \end{center}
\end{figure}

\begin{figure}[ht]
    \begin{center}
      \includegraphics[width=\textwidth]{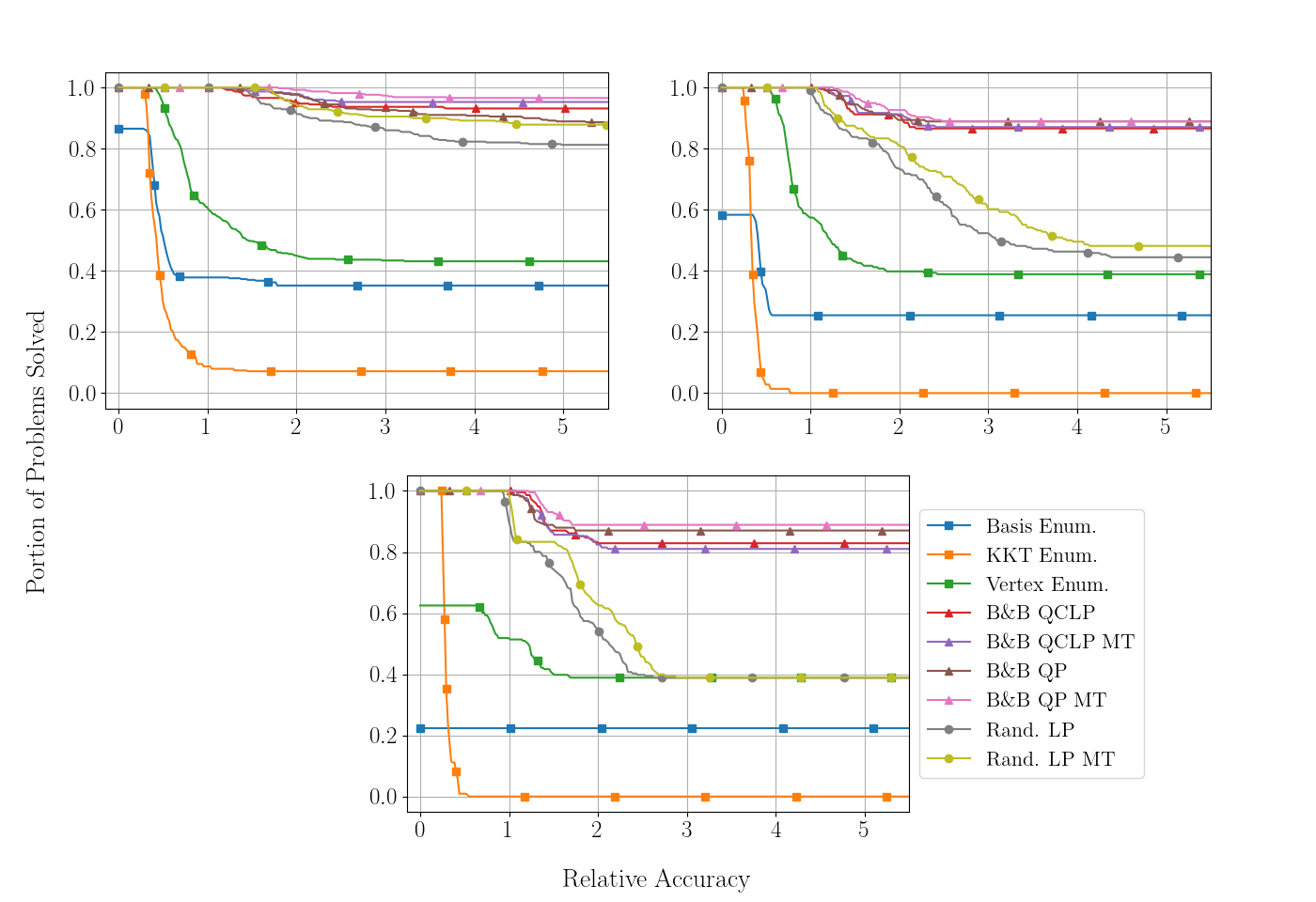}
      \caption{Clockwise, starting from the top left. The accuracy profile of test sets with dimension $n < 30$, $30 \leq n < 70$, and $70 \leq n \leq 100$. \label{fig:dimension_full}}
    \end{center}
\end{figure}

\clearpage
\bibliographystyle{elsarticle-num-names} 
\bibliography{cas-refs}

\begin{thebibliography}{21}
\expandafter\ifx\csname natexlab\endcsname\relax\def\natexlab#1{#1}\fi
\providecommand{\url}[1]{\texttt{#1}}
\providecommand{\href}[2]{#2}
\providecommand{\path}[1]{#1}
\providecommand{\DOIprefix}{doi:}
\providecommand{\ArXivprefix}{arXiv:}
\providecommand{\URLprefix}{URL: }
\providecommand{\Pubmedprefix}{pmid:}
\providecommand{\doi}[1]{\href{http://dx.doi.org/#1}{\path{#1}}}
\providecommand{\Pubmed}[1]{\href{pmid:#1}{\path{#1}}}
\providecommand{\bibinfo}[2]{#2}
\ifx\xfnm\relax \def\xfnm[#1]{\unskip,\space#1}\fi
\bibitem[{Audet and Hare(2017)}]{audet2017derivative}
\bibinfo{author}{C.~Audet}, \bibinfo{author}{W.~Hare}, \bibinfo{title}{Derivative-free and blackbox optimization}, \bibinfo{publisher}{Springer}, \bibinfo{address}{Cham}, \bibinfo{year}{2017}.
\bibitem[{Alarie et~al.(2021)Alarie, Audet, Gheribi, Kokkolaras, and {Le Digabel}}]{directsearchSurvey2021}
\bibinfo{author}{S.~Alarie}, \bibinfo{author}{C.~Audet}, \bibinfo{author}{A.~E. Gheribi}, \bibinfo{author}{M.~Kokkolaras}, \bibinfo{author}{S.~{Le Digabel}},
\newblock \bibinfo{title}{Two decades of blackbox optimization applications},
\newblock \bibinfo{journal}{EURO J. on Comput. Optim.} \bibinfo{volume}{9} (\bibinfo{year}{2021}) \bibinfo{pages}{100011}. \URLprefix \url{https://www.sciencedirect.com/science/article/pii/S2192440621001386}.
\bibitem[{Torczon(1997)}]{torczon_convergence_1997}
\bibinfo{author}{V.~Torczon},
\newblock \bibinfo{title}{On the {convergence} of {pattern} {search} {algorithms}},
\newblock \bibinfo{journal}{SIAM J. on Optim.} \bibinfo{volume}{7} (\bibinfo{year}{1997}) \bibinfo{pages}{1--25}. \URLprefix \url{https://doi.org/10.1137/S1052623493250780}.
\bibitem[{Audet and Dennis(2006)}]{audet2006mads}
\bibinfo{author}{C.~Audet}, \bibinfo{author}{J.~E. Dennis},
\newblock \bibinfo{title}{Mesh adaptive direct search algorithms for constrained optimization},
\newblock \bibinfo{journal}{SIAM J. on Optim.} \bibinfo{volume}{17} (\bibinfo{year}{2006}) \bibinfo{pages}{188--217}. \URLprefix \url{https://doi.org/10.1137/040603371}.
\bibitem[{Dodangeh et~al.(2016)Dodangeh, Vicente, and Zhang}]{dodangeh_optimal_2016}
\bibinfo{author}{M.~Dodangeh}, \bibinfo{author}{L.~N. Vicente}, \bibinfo{author}{Z.~Zhang},
\newblock \bibinfo{title}{On the optimal order of worst case complexity of direct search},
\newblock \bibinfo{journal}{Optim. Lett.} \bibinfo{volume}{10} (\bibinfo{year}{2016}) \bibinfo{pages}{699--708}. \URLprefix \url{https://doi.org/10.1007/s11590-015-0908-1}.
\bibitem[{Jones and McPartlon(2020)}]{sphericalDiscJones2020}
\bibinfo{author}{C.~Jones}, \bibinfo{author}{M.~McPartlon}, \bibinfo{title}{Spherical discrepancy minimization and algorithmic lower bounds for covering the sphere}, \bibinfo{year}{2020}, pp. \bibinfo{pages}{874--891}. \URLprefix \url{https://epubs.siam.org/doi/abs/10.1137/1.9781611975994.53}.
\bibitem[{Hare and Jarry-Bolduc(2020)}]{hare_deterministic_2020}
\bibinfo{author}{W.~Hare}, \bibinfo{author}{G.~Jarry-Bolduc},
\newblock \bibinfo{title}{A deterministic algorithm to compute the cosine measure of a finite positive spanning set},
\newblock \bibinfo{journal}{Optim. Lett.} \bibinfo{volume}{14} (\bibinfo{year}{2020}) \bibinfo{pages}{1305--1316}. \URLprefix \url{https://link.springer.com/10.1007/s11590-020-01587-y}.
\bibitem[{Regis(2021)}]{regis_properties_2021}
\bibinfo{author}{R.~G. Regis},
\newblock \bibinfo{title}{On the properties of the cosine measure and the uniform angle subspace},
\newblock \bibinfo{journal}{Comput. Optim. and Appl.} \bibinfo{volume}{78} (\bibinfo{year}{2021}) \bibinfo{pages}{915--952}. \URLprefix \url{http://link.springer.com/10.1007/s10589-020-00253-4}.
\bibitem[{Regis(2016)}]{regis_properties_2016}
\bibinfo{author}{R.~G. Regis},
\newblock \bibinfo{title}{On the properties of positive spanning sets and positive bases},
\newblock \bibinfo{journal}{Optim. Eng.} \bibinfo{volume}{17} (\bibinfo{year}{2016}) \bibinfo{pages}{229--262}. \URLprefix \url{http://link.springer.com/10.1007/s11081-015-9286-x}.
\bibitem[{Conn et~al.(2009)Conn, Scheinberg, and Vicente}]{DFO-textbook-2009}
\bibinfo{author}{A.~R. Conn}, \bibinfo{author}{K.~Scheinberg}, \bibinfo{author}{L.~Vicente}, \bibinfo{title}{Introduction to Derivative-Free Optimization}, \bibinfo{publisher}{Society for Industrial and Applied Mathematics}, \bibinfo{year}{2009}. \URLprefix \url{https://epubs.siam.org/doi/abs/10.1137/1.9780898718768}.
\bibitem[{Nævdal(2019)}]{naevdal_positive_2019}
\bibinfo{author}{G.~Nævdal},
\newblock \bibinfo{title}{Positive bases with maximal cosine measure},
\newblock \bibinfo{journal}{Optim. Lett.} \bibinfo{volume}{13} (\bibinfo{year}{2019}) \bibinfo{pages}{1381--1388}. \URLprefix \url{http://link.springer.com/10.1007/s11590-018-1334-y}.
\bibitem[{Hare et~al.(2023)Hare, Jarry-Bolduc, and Planiden}]{hare_nicely_2023}
\bibinfo{author}{W.~Hare}, \bibinfo{author}{G.~Jarry-Bolduc}, \bibinfo{author}{C.~Planiden},
\newblock \bibinfo{title}{Nicely structured positive bases with maximal cosine measure},
\newblock \bibinfo{journal}{Optim. Lett.} \bibinfo{volume}{17} (\bibinfo{year}{2023}) \bibinfo{pages}{1495--1515}. \URLprefix \url{https://link.springer.com/10.1007/s11590-023-01973-2}.
\bibitem[{Bartlett(1951)}]{rank1update}
\bibinfo{author}{M.~S. Bartlett},
\newblock \bibinfo{title}{{An Inverse Matrix Adjustment Arising in Discriminant Analysis}},
\newblock \bibinfo{journal}{Ann. Math. Stat.} \bibinfo{volume}{22} (\bibinfo{year}{1951}) \bibinfo{pages}{107 -- 111}. \URLprefix \url{https://doi.org/10.1214/aoms/1177729698}.
\bibitem[{Mangasarian and Shiau(1986)}]{mangasarian_norm-complexity_1986}
\bibinfo{author}{O.~L. Mangasarian}, \bibinfo{author}{T.~H. Shiau},
\newblock \bibinfo{title}{A {variable}-{complexity} {norm} {maximization} {problem}},
\newblock \bibinfo{journal}{SIAM. J. on Algebraic and Discrete Methods} \bibinfo{volume}{7} (\bibinfo{year}{1986}) \bibinfo{pages}{455--461}. \URLprefix \url{https://doi.org/10.1137/0607052}.
\bibitem[{Audet et~al.(2024)Audet, Hare, and Jarry-Bolduc}]{audet2024cosinemeasurerelativesubspace}
\bibinfo{author}{C.~Audet}, \bibinfo{author}{W.~Hare}, \bibinfo{author}{G.~Jarry-Bolduc}, \bibinfo{title}{The cosine measure relative to a subspace}, \bibinfo{year}{2024}. \URLprefix \url{https://arxiv.org/abs/2401.09609}.
\bibitem[{{IBM}(2022)}]{CPLEX}
\bibinfo{author}{{IBM}}, \bibinfo{title}{{IBM ILOG CPLEX 22.1.2 User's Manual}}, \bibinfo{year}{2022}. \URLprefix \url{https://www.ibm.com/docs/en/icos/22.1.2?topic=optimizers-users-manual-cplex}.
\bibitem[{{Gurobi Optimization, LLC}(2024)}]{gurobi}
\bibinfo{author}{{Gurobi Optimization, LLC}}, \bibinfo{title}{{Gurobi Optimizer Reference Manual}}, \bibinfo{year}{2024}. \URLprefix \url{https://www.gurobi.com}.
\bibitem[{Avis(2000)}]{Avis2000-LRS-lexicographical-reverse-search}
\bibinfo{author}{D.~Avis}, \bibinfo{title}{A revised implementation of the reverse search vertex enumeration algorithm}, \bibinfo{publisher}{Birkh{\"a}user Basel}, \bibinfo{address}{Basel}, \bibinfo{year}{2000}, pp. \bibinfo{pages}{177--198}. \URLprefix \url{https://doi.org/10.1007/978-3-0348-8438-9_9}.
\bibitem[{Hare et~al.(2024)Hare, Jarry-Bolduc, Kerleau, and Royer}]{hare-k-spanning-2024}
\bibinfo{author}{W.~Hare}, \bibinfo{author}{G.~Jarry-Bolduc}, \bibinfo{author}{S.~Kerleau}, \bibinfo{author}{C.~W. Royer},
\newblock \bibinfo{title}{Using orthogonally structured positive bases for constructing positive k-spanning sets with cosine measure guarantees},
\newblock \bibinfo{journal}{Linear Algebra Appl.} \bibinfo{volume}{680} (\bibinfo{year}{2024}) \bibinfo{pages}{183--207}. \URLprefix \url{https://www.sciencedirect.com/science/article/pii/S0024379523003798}.
\bibitem[{Jarry-Bolduc et~al.(2020)Jarry-Bolduc, Nadeau, and Singh}]{jarry-bolduc_uniform_2020}
\bibinfo{author}{G.~Jarry-Bolduc}, \bibinfo{author}{P.~Nadeau}, \bibinfo{author}{S.~Singh},
\newblock \bibinfo{title}{Uniform simplex of an arbitrary orientation},
\newblock \bibinfo{journal}{Optim. Lett.} \bibinfo{volume}{14} (\bibinfo{year}{2020}) \bibinfo{pages}{1407--1417}. \URLprefix \url{http://link.springer.com/10.1007/s11590-019-01448-3}.
\bibitem[{Davis(1954)}]{davis-positive-independence}
\bibinfo{author}{C.~Davis},
\newblock \bibinfo{title}{Theory of positive linear dependence},
\newblock \bibinfo{journal}{Am. J. Math.} \bibinfo{volume}{76} (\bibinfo{year}{1954}) \bibinfo{pages}{733--746}. \URLprefix \url{http://www.jstor.org/stable/2372648}.

\end{thebibliography}



\end{document}